\documentclass[12pt]{amsart}

\usepackage[ textheight=615pt,
  textwidth=360pt,
  centering
]{geometry}

\usepackage{xcolor}
\usepackage[english]{babel}
\usepackage[T1]{fontenc}
\usepackage[utf8]{inputenc}

\usepackage{amsmath,amssymb,amsthm,mathrsfs}
\usepackage{tikz-cd}
\usepackage{graphicx}
\usepackage{xypic}
\usepackage{hyperref}

\hypersetup{
  colorlinks=true,
  linkcolor=blue,
  urlcolor=red,
  pdftitle={}
}

\usepackage{dsfont}

\usepackage{graphicx}
\usepackage{xypic}
\usepackage{enumitem}
\usepackage{datetime}
\usepackage{xcolor}

\usepackage{mathtools}

\usepackage{hyperref} 
\hypersetup{
    colorlinks=true,
    linkcolor=blue,
    urlcolor=red,
    pdftitle={},
    }

\newtheorem{theorem}{Theorem}[section]
\newtheorem*{theorem*}{Theorem}
\newtheorem{proposition}[theorem]{Proposition}

\newtheorem{lemma}[theorem]{Lemma}
\newtheorem{conjecture}[theorem]{Conjecture}

\newtheorem{corollary}[theorem]{Corollary}
\newtheorem{definition}[theorem]{Definition}

\newtheorem{remark}[theorem]{Remark}

\newcommand{\C}{\mathbb{C}}

\newcommand{\Z}{\mathbb{Z}}

\newcommand{\R}{\mathbb{R}}

\newcommand{\X}{\mathcal{X}}
\newcommand{\T}{\mathbb{T}}

\newcommand{\Q}{\mathbb{Q}}

\newcommand{\LL}{\mathcal{L}}

\begin{document}
\title{Arithmetic theta invariants and arithmetic equilibrium measures
}

\author{Mounir Hajli}
\thanks{ }



\address{Westlake University}

\email{hajlimounir@gmail.com}

            \date{
            }
            
            \maketitle 
            
\begin{abstract}

We introduce two new arithmetic invariants associated with Hermitian line bundles on arithmetic varieties, which play the role of the Bergman distortion function in the arithmetic setting. As a first step, we show that these two functions are asymptotically equivalent. This leads to the introduction of an intrinsic measure $
\widehat{\mu}^{\sup}_{\mathrm{eq}},$
which we show to be a natural arithmetic analogue of the equilibrium measure.

We prove that $\widehat{\mu}^{\sup}_{\mathrm{eq}}$ serves as a detector of arithmetic positivity. In particular, we determine it for weakly nef Hermitian line bundles and, in full generality, in the toric case.

Moreover, $\widehat{\mu}^{\sup}_{\mathrm{eq}}$ governs the arithmetic volume function and its variational properties. More precisely, we establish weak integral representations of the arithmetic volume and of its first variation in terms of $\widehat{\mu}^{\sup}_{\mathrm{eq}}$. These formulas provide a measure-theoretic interpretation of the volume derivative of Yuan--Zhang and reveal a structural parallel with variational formulas in complex pluripotential theory.

Our approach differs fundamentally from previous treatments based on Harder--Narasimhan filtrations or arithmetic Okounkov bodies. It places the arithmetic volume at the center of the theory and suggests that asymptotic volume invariants, rather than heights, provide the natural analytic framework for problems in equidistribution and arithmetic dynamics.

\end{abstract}

\section{Introduction}

The Bergman kernel plays a central role in complex geometry and pluripotential
theory, encoding the asymptotic behavior of spaces of holomorphic sections of
high tensor powers of a line bundle, \cite{Bouche, Catlin, Tian, Zelditch}.  From a variational viewpoint, Berman \cite{Berman2009} proved that, for
$\phi \in \mathscr{C}^2$ a weight on a holomorphic line bundle $L$ on a complex manifold $X$,  the rescaled Bergman measure converges weakly:
\begin{equation}\label{B1}
\lim_{k \to \infty} \frac{1}{k^n} \rho(\mu, k\phi)\, \mu = \mu_\phi,
\end{equation}
This convergence provides an intrinsic integral representation for the volume of $L$, expressed as 
\begin{equation}\label{B2}
\mathrm{vol}(L) = n! \int_X \mu_\phi, \qquad \mathrm{vol}(L) := \limsup_{k \to \infty} \frac{\dim_\mathbb{C} H^0(X, kL)}{k^n/n!}.
\end{equation}
This framework identifies the equilibrium measure as the natural density governing both asymptotic dimension growth and the variational properties of the volume functional. Subsequent developments, notably by Hisamoto and Finski
\cite{finski2024wess_v2, Hisamoto1, Hisamoto2}, have extended this perspective to
restricted and weighted Bergman kernels.

\medskip

\medskip

In Arakelov geometry, several notions of arithmetic positivity have been introduced and developed, leading to deep applications in Diophantine geometry; see for instance \cite{Moriwaki2, YuanInventiones, Zha95}. A central problem is to obtain effective and intrinsic criteria for arithmetic positivity and to understand the asymptotic behavior of arithmetic Hilbert functions and arithmetic volumes.

In the toric setting, substantial progress has been achieved. In a series of works, Burgos, Moriwaki, Philippon, and Sombra  gave 
 characterizations of arithmetic ampleness, nefness, and bigness together with   explicit formulae for the arithmetic volume and the $\chi$-arithmetic volume of toric Hermitian line bundles  ; see \cite{Burgos3, Moriwaki}. These results reveal the strength of convex-geometric and analytic techniques in  arithmetic toric geometry.

\medskip

However, in contrast with the complex setting, the arithmetic framework does not admit a direct analogue of the Bergman kernel asymptotics that underlie Berman’s variational approach. The asymptotic study of arithmetic Hilbert functions and arithmetic volumes has therefore relied primarily on combinatorial and slope-theoretic constructions, such as arithmetic Okounkov bodies and Harder--Narasimhan filtrations of direct image sheaves \cite{Chen1, Chen-2011, YuanInventiones, Yuan_2009}. While powerful and far-reaching, these methods depend on auxiliary choices, flags, filtrations, or slope decompositions, and do not arise from an intrinsic analytic variational structure.

\medskip

The primary purpose of this paper is to develop a fully intrinsic analytic approach to arithmetic volume and positivity, inspired by the variational methods of complex geometry but adapted to the arithmetic setting. 

More precisely, let $\X$ be a smooth projective arithmetic variety over $\mathbb{Z}$ of dimension $n+1$ and let $\overline{\LL}_\phi$ be a continuous Hermitian line bundle on $\X$, where $\phi$ is a continuous weight on $\LL$ defining the metric. Fixing a smooth volume form on $\X(\mathbb{C})$ induces a Euclidean structure on the lattice of global sections of $\LL$. 
 To any sublattice $E$ of this lattice, we associate two functions on 
$\mathcal{X}(\mathbb{C})$, denoted by
\[
\Theta((\mathcal{X},\mu),\overline E)
\quad \text{and} \quad
\rho(\mu,\overline E).
\]
The first is called the \emph{arithmetic distortion function} attached to $\overline E$. 
It can be interpreted as the expectation of the pointwise norm with respect to the Gaussian probability measure induced by the $L^2$-norm on $E$. 
The second function, $\rho(\mu,\overline E)$, measures the distortion between the $L^2$-norm and the supremum norm (see Definition~\ref{defTheta}). 
When $E$ is the full lattice, $\rho(\mu,\overline E)$ coincides with the geometric distortion function $\rho(\mu,\phi)$, and we denote the corresponding arithmetic distortion function simply by $\Theta(\mu,\phi)$.

 At first glance, one might expect $\rho(\mu,\overline E)$ to be a more natural candidate than $\Theta((\mathcal{X},\mu),\overline E)$. 
However, we shall show that the latter is better suited for asymptotic analysis and exhibits greater flexibility. 
Nevertheless, we will explain (see \eqref{eq:13_3}) that, asymptotically, the two functions do not differ significantly.

.

\medskip

We prove that these invariants satisfy the fundamental inequality
\begin{equation}\label{eq:theta-leq-rho}
\Theta(\mu,\overline E)\leq \rho(\mu,\overline E),
\end{equation}
which generalizes  \cite[Theorem~4.7]{HajliTAMS}.

By analogy with the geometric distortion function, it is natural to ask whether $\Theta(\mu,k\phi)$ admits a full asymptotic expansion as $k\to\infty$. A positive answer would imply the existence of an asymptotic expansion for certain arithmetic Hilbert functions and would provide a new variational description of arithmetic volume.
This leads naturally to two asymptotic arithmetic measures attached to a Hermitian line bundle $\overline{\LL}_\phi$ on an arithmetic variety $\X$:
\[
\widehat{\mu}^{\sup}_{\mathrm{eq}}((\X,\mu),\overline{E}_\bullet)
:=
\limsup_{k \to \infty}
\frac{1}{k^n}\Theta(\mu,\overline E_k)\,\mu,\]
\[
\quad
\widehat{\omega}^{\sup}_{\mathrm{eq}}((\X,\mu),\overline{E}_\bullet)
:=
\limsup_{k \to \infty}
\frac{1}{k^n}\rho(\mu,\overline E_k)\,\mu,
\]
where $\overline E_k$ are sublattices of $H^0(\X,k\LL^{})$. 
When $E_k$ is the full lattice of global sections, we write
\[
\widehat{\mu}^{\sup}_{\mathrm{eq}}((\X,\mu),\overline{\LL}_\phi)
=
\limsup_{k \to \infty}
\frac{1}{k^n}\Theta(\mu,k\phi)\,\mu,
\]
and similarly for $\widehat{\omega}^{\sup}_{\mathrm{eq}}$. We call  $\widehat{\mu}^{\sup}_{\mathrm{eq}}((\X,\mu),\overline{\LL}_\phi)$ the \emph{arithmetic equilibrium invariant}.  We define 
$\widehat{\mu}^{\inf}_{\mathrm{eq}}(\cdot)$ and 
$\widehat{\omega}^{\inf}_{\mathrm{eq}}(\cdot)$
by substituting $\liminf$ for $\limsup$ in the preceding definitions.

\medskip

The analysis of how sharp inequality~\eqref{eq:theta-leq-rho} is at the level of asymptotics led to the introduction of a natural arithmetic graded object 
$\overline{S}_{\phi,\bullet}$, consisting of the submodules generated by the strictly small sections of $\overline{\mathcal{L}}_\phi$. 
For each $k$, we denote by $\overline{S}_{\phi,k}$ the $\mathbb{Z}$-module $S_{\phi,k}$ endowed with the norm induced from 
$\overline{H^0(\mathcal{X},k\mathcal{L})}_{(L^2,k\phi)}$. This construction leads to a refined asymptotic inequality, clarifying the structural role of the arithmetic equilibrium invariant:
\begin{equation}\label{eq:13_3}
\widehat{\omega}^{\sup}_{\mathrm{eq}}\big((\mathcal{X},\mu), \overline{S}_{\phi,\bullet}\big)
\;\leq\;
\widehat{\mu}^{\sup}_{\mathrm{eq}}\big((\mathcal{X},\mu), \overline{\mathcal{L}}_\phi\big)
\;\leq\;
\widehat{\omega}^{\sup}_{\mathrm{eq}}\big((\mathcal{X},\mu), \overline{S}_{\phi+t,\bullet}\big),
\quad \forall\, t>0.
\end{equation}

This shows that the invariant 
$\widehat{\mu}^{\sup}_{\mathrm{eq}}\big((\mathcal{X},\mu), \overline{\mathcal{L}}_\phi\big)$
is intrinsically arithmetic in nature, see Theorem~\ref{thm:enveloppe}. In toric setting, we  prove a refinement of \eqref{eq:13_3}. Namely, we  show
\[
\widehat{\omega}^{\sup}_{\mathrm{eq}}\left((\mathcal X,\mu),\overline{S}_{\phi,\bullet}\right)
=
\widehat{\mu}^{\sup}_{\mathrm{eq}}\left((\mathcal X,\mu),\overline{\mathcal L}_\phi\right),
\]
see Theorem \ref{thm:13_2}.\\

\medskip

Next, we explain the interplay between the arithmetic equilibrium measure and arithmetic positivity (Definitions \ref{def:arith-pos1} and \ref{def:weak-nef}).
 We have  the fundamental inequality
\[
\widehat{\mu}^{\sup}_{\mathrm{eq}}((\X,\mu),\overline{\LL}_\phi)
\le
\widehat{\omega}^{\sup}_{\mathrm{eq}}((\X,\mu),\overline{\LL}_\phi),
\] 
which follows easily from \eqref{eq:theta-leq-rho}. 
If $\phi \in \mathscr{C}^2$, then $
\widehat{\omega}^{\sup}_{\mathrm{eq}}((\X,\mu),\overline{\LL}_\phi)
=
\mu_\phi,$
the equilibrium measure associated with $\phi$. Consequently,
\begin{equation}\label{eq1}
\widehat{\mu}^{\sup}_{\mathrm{eq}}((\X,\mu),\overline{\LL}_\phi)
\le
\mu_\phi.
\end{equation}

Our first main result shows that under natural positivity assumptions this inequality becomes an equality.

\begin{theorem}[cf.~Theorem~\ref{main}]
\label{thm:23_1}
If $\overline{\LL}_\phi$ is weakly nef and $\LL$ is ample, then
\begin{equation}\label{case1}
\lim_{k \to \infty}
\frac{1}{k^n}\Theta(\mu,k\phi)\,\mu
=
\mu_\phi.
\end{equation}
In particular, $\widehat{\mu}^{\sup}_{\mathrm{eq}}((\X,\mu),\overline{\LL}_\phi)=\mu_\phi$.
\end{theorem}

The connection between arithmetic bigness and the arithmetic equilibrium measure is illustrated in Theorem~\ref{thm:big}, where we show that arithmetic bigness can be characterized by a uniform positive lower bound on 
$\widehat{\mu}^{\inf}_{\mathrm{eq}}((\mathcal{X},\mu),\overline{\mathcal{L}}_\phi)$.

On the other hand, if $\mathcal{L}_{\mathbb{Q}}$ is big while the asymptotic maximal slope 
$\widehat{\mu}_{\max}(\overline{\mathcal{L}}_\phi)$ is negative, then the equilibrium measure $\mu_\phi$ is non-zero, yet
\[
\widehat{\mu}^{\sup}_{\mathrm{eq}}\big((\mathcal{X},\mu),\overline{\mathcal{L}}_\phi\big)
=
0
\quad \text{on } \mathcal{X}(\mathbb{C}),
\]
see Proposition~\ref{vanishingTheta}. 
In this situation, the inequality~\eqref{eq1} is therefore strict. 
This demonstrates that $\widehat{\mu}^{\sup}_{\mathrm{eq}}$ detects arithmetic positivity in a sharp manner.\\

\medskip

In the toric setting, the same conclusions hold under milder assumptions. 
For instance, Theorem~\ref{thm:23_1} remains valid for nef toric Hermitian line bundles on a smooth toric variety over $\mathbb{Z}$; see Theorem~\ref{toricThetaNef}. 

The geometric meaning of $\widehat{\mu}^{\sup}_{\mathrm{eq}}$ becomes particularly transparent in the toric case. 
Let $\overline{\mathcal{L}}_\phi$ be a smooth toric Hermitian line bundle on a smooth toric variety $\mathcal{X}$ with $c_1(\overline{\mathcal{L}}_\phi)$ positive, and set
\[
\mu = c_1(\overline{\mathcal{L}}_\phi)^n.
\]
We then establish the following result.

\begin{theorem}[cf.~Theorem~\ref{thm:theta-toric-1}]\label{thm:73_1}
There exists a measurable subset $U \subset \mathcal{X}(\mathbb{C})$ such that
\[
\lim_{k \to \infty}
\frac{1}{k^n}\Theta(\mu,k\phi)\,\mu
=
\mathds{1}_U\,\mu
\quad \text{weakly}.
\]
\end{theorem}

\medskip

The robustness of the preceding results motivates the following conjecture.

\begin{conjecture}\label{conj1}
Let $(\mathcal{X},\mu)$ be a smooth projective arithmetic variety over $\mathbb{Z}$ endowed with a smooth volume form $\mu$, and let 
$\overline{\mathcal{L}}_\phi$ be a smooth Hermitian line bundle on $\mathcal{X}$. 
Then the $\limsup$ in the definition of 
\[
\widehat{\mu}^{\sup}_{\mathrm{eq}}\big((\mathcal{X},\mu),\overline{\mathcal{L}}_\phi\big)
\]
is in fact a limit. Moreover, this limit is independent of the choice of the auxiliary volume form $\mu$. We denote it by
\[
\widehat{\mu}^{}_{\mathrm{eq}}(\mathcal{X},\overline{\mathcal{L}}_\phi).
\]
\end{conjecture}

\medskip

\medskip

Having formulated Conjecture~\ref{conj1}, it becomes natural to clarify its relationship with the arithmetic volume. Indeed, the arithmetic volume function provides the fundamental asymptotic invariant governing the growth of global sections, and any intrinsic measure capturing arithmetic positivity should ultimately be reflected at the level of volume.

The arithmetic volume function is a central invariant in Arakelov geometry and has been studied extensively over the past decade; see \cite{Chen1, Chen-2011,Moriwaki1, YuanInventiones, Yuan_2009}. The \emph{arithmetic volume} of a Hermitian line bundle $\overline{\LL}_\phi$ is defined by
\[
\widehat{\mathrm{vol}}(\overline{\LL}_\phi)
:=
\limsup_{k\to \infty}
\frac{\hat{h}^0\!\left(\overline{H^0(\X,k\LL)}_{(\sup,k\phi)}\right)}
{k^{n+1}/(n+1)!},
\]
where
$
\hat{h}^0\!\left(\overline{H^0(\X,\LL)}_{(\sup,\phi)}\right)
:=
\log \# \left\{
s \in H^0(\X,\LL)
\mid
\|s\|_{\sup,\phi} \le 1
\right\}.$

In \cite{HajliTAMS}, we introduced the auxiliary quantities
\[
h_\theta^0\!\left(\overline{H^0(\X,k\LL)}_{(L^2,k\phi)}\right)
\quad\text{and}\quad
h_\theta^0\!\left(\overline{H^0(\X,k\LL)}_{(\sup,k\phi)}\right),
\]
which are essentially equivalent to
$\hat{h}^0\!\left(\overline{H^0(\X,k\LL)}_{(L^2,k\phi)}\right)$
and
$\hat{h}^0\!\left(\overline{H^0(\X,k\LL)}_{(\sup,k\phi)}\right)$,
respectively (see Equation~\eqref{defoftheta}).

Motivated by the arithmetic volume, we defined the modified asymptotic invariants
\[
\widehat{\mathrm{vol}}_{L^2,\theta}(\overline{\LL}_\phi)
:=
\limsup_{k\to \infty}
\frac{
h_\theta^0\!\left(\overline{H^0(\X,k\LL)}_{(L^2,k\phi)}\right)
}{
k^{n+1}/(n+1)!
},
\]
and
\[
\widehat{\mathrm{vol}}_{\sup,\theta}(\overline{\LL}_\phi)
:=
\limsup_{k\to \infty}
\frac{
h_\theta^0\!\left(\overline{H^0(\X,k\LL)}_{(\sup,k\phi)}\right)
}{
k^{n+1}/(n+1)!
}.
\]

These formulations are particularly well suited to the analytic approach developed here. First, it is shown in \cite{HajliTAMS} that
\begin{equation}\label{182_1}
\widehat{\mathrm{vol}}(\overline{\LL}_\phi)
=
\widehat{\mathrm{vol}}_{L^2,\theta}(\overline{\LL}_\phi),
\end{equation}
see Theorem~\ref{MorMor} for a concise proof. Theorem \ref{MorMor} shows that the arithmetic volume can be entirely recovered from the asymptotics of the Gaussian theta invariants $h_\theta^0$.

\medskip

Our next result connects the arithmetic volume directly with the arithmetic distortion function $\Theta(\mu,\phi)$ and the asymptotic first minimum
$\widehat{\mu}_{\max}(\overline{\LL}_\phi)$.
This provides a variational representation of arithmetic volume and clarifies the structural meaning of Conjecture~\ref{conj1}.

\begin{theorem}[Weak Representation Theorem \ref{Representation}]\label{RepresentationIntro}
The arithmetic volume satisfies
\begin{equation}\label{WeakConjecture1}
\widehat{\mathrm{vol}}(\overline{\LL}_\phi)
=
(n+1)!
\limsup_{k \to \infty}
\int_0^{\widehat{\mu}_{\max}(\overline{\LL}_\phi)}
\int_{\X(\C)}
\frac{1}{k^n}
\Theta(\mu, k(\phi-t))
\, {{\mu}} \, dt.
\end{equation}
\end{theorem}

Our proof of~\eqref{WeakConjecture1} relies essentially on the identity~\eqref{182_1}. 
It may be viewed as a weak arithmetic analogue of the complex integral representation of the volume in~\eqref{B2}. 
In particular, it shows that the arithmetic distortion function controls the asymptotic growth of small sections, thereby providing a variational interpretation of the arithmetic volume.

On the one hand, an immediate consequence of~\eqref{WeakConjecture1} yields the inequality
\[
\widehat{\mathrm{vol}}(\overline{\mathcal{L}}_\phi)
\leq 
(n+1)!
\int_0^{\widehat{\mu}_{\max}(\overline{\mathcal{L}}_\phi)}
\int_{\mathcal{X}(\mathbb{C})}
\widehat{\mu}_{\mathrm{eq}}(\mathcal{X},\overline{\mathcal{L}}_{\phi-t})
\, \mu \, dt.
\]

\medskip

On the other hand, assuming Conjecture~\ref{conj1}, formula~\eqref{WeakConjecture1} yields the stronger integral representation
\begin{equation}\label{rep3}
\widehat{\mathrm{vol}}(\overline{\mathcal{L}}_\phi)
=
(n+1)!
\int_0^{\widehat{\mu}_{\max}(\overline{\mathcal{L}}_\phi)}
\int_{\mathcal{X}(\mathbb{C})}
\widehat{\mu}_{\mathrm{eq}}(\mathcal{X},\overline{\mathcal{L}}_{\phi-t})
\, \mu \, dt.
\end{equation}

In particular, this identity shows that the $\limsup$ in the definition of the arithmetic volume is in fact a limit, thereby recovering a theorem of Chen \cite{Chen1}.  Under the assumptions and hypothesis of Theorem \ref{thm:73_1}, the arithmetic volume admits a strong integral representation.  

\medskip

The representation \eqref{rep3} differs fundamentally from earlier integral formulae based on arithmetic Okounkov bodies and slope-theoretic constructions \cite{Chen, Burgos3, Yuan_2009, Yuan}. Instead, it arises from an intrinsic analytic variational framework and parallels the geometric philosophy introduced in \cite{Berman2009}.\\

We now further clarify the relationship between our results and earlier work, culminating in a new perspective on equidistribution phenomena.

\medskip

In their study of effective bounds for linear series on arithmetic varieties, Yuan and Zhang \cite{YuanZhang1, YuanZhang2} introduced the \emph{volume derivative}
$d\mathrm{vol}(\overline{\mathcal{L}}_\phi)$ and proved, using arithmetic Fujita approximation, that
\[
d\mathrm{vol}(\overline{\mathcal{L}}_\phi)
=
\frac{1}{n}
\lim_{t \to 0}
\frac{
\widehat{\mathrm{vol}}(\overline{\mathcal{L}}_{\phi+t})
-
\widehat{\mathrm{vol}}(\overline{\mathcal{L}}_\phi)
}{t}.
\]
By the differentiability of arithmetic volume established in \cite{Chen-2011}, this limit exists and is given by a certain intersection number 
$n\,
\langle \overline{\LL}_\phi^{\,n} \rangle
\cdot
\overline{\mathcal{O}}_0,$
where $\overline{\mathcal{O}}_0$ denotes the trivial Hermitian line bundle.

An important property of the invariant $d\mathrm{vol}(\overline{\mathcal{L}}_\phi)$ is that when
$\overline{\mathcal{L}}_\phi$ is nef, it coincides with the degree of
$\mathcal{L}_\mathbb{Q}$.

Our approach provides a different variational interpretation of this derivative.

\begin{theorem}[cf.~Theorem~\ref{corr271}]
Let $\overline{\LL}_\phi$ be a weakly nef Hermitian line bundle on a smooth projective arithmetic variety $\X$ and assume that $\LL$ is ample. Let $\eta$ be a nonnegative smooth weight on $\mathcal{O}$. Then
\[
\lim_{t\to 0^+}
\frac{1}{t}
\left(
\widehat{\mathrm{vol}}(\overline{\LL}_{\phi+t\eta})
-
\widehat{\mathrm{vol}}(\overline{\LL}_\phi)
\right)
=(n+1)!
\int_{\X(\C)} \eta\, \mu_\phi.
\]
\end{theorem}

Under these assumptions, the volume derivative of Yuan and Zhang is therefore expressed as the integral of the equilibrium measure $\mu_\phi$. 

This strongly suggests that, in greater generality, the volume derivative should admit an integral representation in terms of the arithmetic equilibrium measure $\widehat{\mu}^{\sup}_{\mathrm{eq}}(\X,\overline{\LL}_\phi)$. More precisely, assuming Conjecture~\ref{conj1}, Lemma~3.3 of \cite{YuanZhang2} implies
\begin{equation}\label{cc2}
\int_{\X(\C)}
\widehat{\mu}^{}_{\mathrm{eq}}(\X,\overline{\LL}_\phi)
=
d\mathrm{vol}(\overline{\LL}_\phi),
\end{equation}
providing a refined and intrinsic interpretation of the invariant $d\mathrm{vol}(\overline{\LL}_\phi)$.

\medskip

In this direction, and independently of the existing approaches,
we prove that
\[
\widehat{\mathrm{vol}}(\overline{\LL}_\phi)
-
\widehat{\mathrm{vol}}(\overline{\LL}_\psi)
=(n+1)!
\int_{\X(\C)}
(\phi-\psi)
\left(
\int_0^1
\widehat{\mu}^{\sup}_{\mathrm{eq}}
\bigl(
\X,
\overline{\LL}_{\phi_t}
\bigr)
\, dt
\right),
\]
where $\phi_t=t\psi+(1-t)\phi$ for $t\in[0,1]$, and where $\phi$ and $\psi$ are smooth weights such that
$\overline{\LL}_\phi$ and $\overline{\LL}_\psi$ are nef and $\LL$ is ample
(see the proof of Theorem~\ref{corr271}).
Conjecture~\ref{conj1} would extend this identity to full generality.

This identity provides a precise variational formula for the arithmetic volume
with respect to a change of metric.
It may be viewed as an  analogue of the classical variation formula
for heights:
\[
h_{\overline{\mathcal{L}}_\phi}(\X)
-
h_{\overline{\mathcal{L}}_\psi}(\X)
=
\int_{\X(\C)}
(\phi-\psi)
\left(
\int_0^1 \mathrm{MA}(\phi_t)\, dt
\right),
\]
where $\mathrm{MA}(\phi)$ denotes the Monge--Amp\`ere measure associated with $\phi$
(cf.~\cite[Proposition~3.2.2]{BoGS}).
When the metrics are semipositive, $\mathrm{MA}(\phi)$ coincides with the corresponding equilibrium measure.
In this sense, the invariant $\widehat{\mu}^{\sup}_{\mathrm{eq}}$ plays the role of an arithmetic
Monge--Amp\`ere density governing the variation of the volume.

\medskip

We emphasize that 
\eqref{cc2}
 is  established under the assumptions that
$\LL$ is ample and $\overline{\LL}_\phi$ is weakly nef
(see Theorem~\ref{corr271}).
These hypotheses are natural and parallel those appearing in Yuan's
equidistribution theorem for points of small height \cite{YuanInventiones}.

\medskip

We conclude this introduction by explaining how the theory developed here
leads naturally to a new perspective on the equidistribution of small points.
In particular, we show that the arithmetic distortion function  and the
arithmetic volume function provide a conceptually intrinsic framework for
this problem.

\medskip

We begin with a brief review of the general setting. Given $\eta \in \mathbb{R}$, denote by 
$\X(\overline{\mathbb{Q}})_{\leq \eta}$ the set of algebraic points 
$x \in \X(\overline{\mathbb{Q}})$ such that 
$\mathrm{h}_{\overline{\LL}_\phi}(x) \leq \eta$. 
The \emph{essential minimum} of $\X$ with respect to 
$\overline{\LL}_\phi$ is defined as
\[
\mu^{\mathrm{ess}}_{\overline{\LL}_\phi}(\X)
=
\inf \left\{ \eta \in \mathbb{R}
\;\middle|\;
\X(\overline{\mathbb{Q}})_{\le \eta}
\text{ is Zariski dense in } \X
\right\}.
\]

Zhang's fundamental inequality \cite{Zha95} asserts that
\begin{equation}\label{eq:zhang_bound}
\frac{h_{\overline{\LL}_\phi}(\X)}
{(n+1)\deg_{\LL_\mathbb{Q}}(\X_\mathbb{Q})}
\leq 
\mu_{\overline{\LL}_\phi}^{\mathrm{ess}}(\X).
\end{equation}
On the other hand, the asymptotic maximal slope
$\widehat{\mu}_{\max}(\overline{\LL}_\phi)$ satisfies
\[
\widehat{\mu}_{\max}(\overline{\LL}_\phi)
\leq
\mu_{\overline{\LL}_\phi}^{\mathrm{ess}}(\X),
\]
together with the volume inequality
\begin{equation}\label{eq:vol-mu}
\frac{\widehat{\mathrm{vol}}(\overline{\LL}_\phi)}
{(n+1)\,\mathrm{vol}(\LL_\mathbb{Q})}
\leq
\widehat{\mu}_{\max}(\overline{\LL}_\phi).
\end{equation}

More recently, Burgos--Philippon--Sombra \cite{Burgos_distribution}
introduced the notion of a \emph{quasi-canonical} Hermitian line bundle,
for which equality holds in \eqref{eq:zhang_bound}, in their study of the
distribution of Galois orbits of small points on proper toric varieties.
They obtained a complete description of equidistribution in that setting.

\medskip

In the context of the present work, and still under mild positivity
assumptions, equality in \eqref{eq:zhang_bound}
implies equality in \eqref{eq:vol-mu}.
This observation plays a structural role in our approach.
Rather than working primarily with the height function,
we place the arithmetic volume function at the center of the theory.

Classically, height theory has been the fundamental tool in
equidistribution problems.
This is largely due to the flexibility of heights under metric variation,
which is crucial in perturbation arguments leading to equidistribution
theorems (see, for example,
\cite{BermanBoucksom, CL2006, SUZ}).
However, the theory of heights relies on the full machinery of
arithmetic intersection theory
\cite{BoGS, AIT, Character, Character2},
later reformulated in the adelic framework by Zhang \cite{Zhang}
and extended by Chambert-Loir and Gubler
\cite{Gubler2003, CL2006}.

In contrast, the arithmetic volume function may at first appear too rigid
for such applications.
One of the main messages of this paper is that this rigidity is deceptive:
the asymptotic structure of the arithmetic volume, together with the
arithmetic distortion invariants introduced here,
encodes precisely the analytic information required for equidistribution.
In this sense, the distribution of Galois orbits is intrinsically governed
by the asymptotic behavior of the arithmetic volume function.

\medskip

We illustrate this perspective in the setting of arithmetic dynamics.
Let $(\mathcal{X}, f, \mathcal{L})$ be an algebraic dynamical system over
$\mathbb{Z}$, and let
$\overline{\mathcal{L}}_{\phi_\infty}$ denote the canonical metric
characterized by
\[
f^\ast \overline{\mathcal{L}}_{\phi_\infty}
\simeq
\overline{\mathcal{L}}_{\phi_\infty}^{\otimes d}.
\]
Assume that $\LL$ is ample.
For a generic sequence of small points $(x_m)_{m \in \mathbb{N}}$
with respect to $\overline{\mathcal{L}}_{\phi_\infty}$,
let $\mu_{x_m}$ be the probability measures supported on
their Galois orbits.

Using the framework developed in this paper,
based on the asymptotics of arithmetic distortion functions
and the arithmetic volume function,
we recover Yuan's equidistribution theorem, in the case when $\LL$ is ample, in an intrinsic and analytic manner:
\[
\lim_{m \to \infty}
\int_{\mathcal{X}(\mathbb{C})} g \, {{\mu}}_{x_m}
=
\int_{\mathcal{X}(\mathbb{C})} g \, {{\mu}}_{\phi_\infty},
\qquad
\forall g \in \mathscr{C}^0(\mathcal{X}(\mathbb{C})).
\]

This reformulation shows that equidistribution is naturally governed
by the variational structure of the arithmetic volume.
It provides a conceptual alternative to height-based arguments
and suggests that the arithmetic volume function
is the fundamental asymptotic invariant underlying
the distribution of small points.






{}
\section{Distortion functions, equilibrium measures, and geometric volume}
\label{sec2}

In this section we recall the analytic notions that will serve as a model
for the arithmetic theory developed later.

\medskip

Let $X$ be a compact complex manifold of dimension $n$, and let $\mu$
be a smooth positive volume form on $X$.
Let $L$ be a holomorphic line bundle on $X$.

A \emph{weight} $\phi$ on $L$ is a locally integrable function on
the complement of the zero section of $L^\ast$ satisfying the
log-homogeneity condition
\[
\phi(\lambda v) = \log|\lambda| + \phi(v)
\quad \text{for } v \in L^\ast \setminus \{0\},\ \lambda \in \C^\ast.
\]
Such a weight defines a Hermitian metric on $L$, denoted by
$\|\cdot\|_\phi$.

\medskip

If $\phi$ is continuous, the space of global sections $H^0(X,L)$
carries the $L^2$-norm
\[
\|s\|_{(L^2,\phi)}^2
:=
\int_X \|s(x)\|_\phi^2 \, {{\mu}},
\quad s \in H^0(X,L),
\]
as well as the sup-norm
\[
\|s\|_{\sup,\phi}
:=
\sup_{x \in X} \|s(x)\|_\phi.
\]

\medskip

The \emph{Bergman distortion function} associated to $(\mu,\phi)$ is
\begin{equation}\label{distortion}
\rho(\mu,\phi)(x)
:=
\sup_{s \in H^0(X,L)\setminus\{0\}}
\frac{\|s(x)\|_\phi^2}{\|s\|_{(L^2,\phi)}^2}.
\end{equation}
If $\{s_1,\dots,s_N\}$ is an $(L^2,\phi)$-orthonormal basis of
$H^0(X,L)$, then
\[
\rho(\mu,\phi)(x)
=
\sum_{j=1}^N \|s_j(x)\|_\phi^2.
\]

A fundamental property is
\begin{equation}\label{intRho}
\int_X \rho(\mu,\phi)\, {{\mu}}
=
\dim_\C H^0(X,L).
\end{equation}

\begin{definition}\label{defBM}
We say that $\mu$ satisfies the \emph{Bernstein--Markov property}
with respect to $\|\cdot\|_\phi$ if for every $\varepsilon>0$,
\[
\sup_X \rho(\mu,k\phi)^{1/2}
=
O(e^{k\varepsilon})
\quad \text{as } k\to\infty.
\]
\end{definition}

If $\mu$ is smooth and positive and $\phi$ is continuous, then
$\mu$ satisfies the Bernstein--Markov property
\cite[Lemma~3.2]{BermanBoucksom}.

\medskip

Assume now that $L$ is big.
For smooth weights $\phi$ and $\psi$ on $L$, define the
Monge--Amp\`ere measure
\[
\mathrm{MA}(\phi) := (dd^c \phi)^n.
\]
The Monge--Amp\`ere energy functional $\mathcal E$ is defined by
\[
\mathcal E(\overline L_\phi)
-
\mathcal E(\overline L_\psi)
=
\frac{1}{n+1}
\sum_{j=0}^n
\int_X
(\phi-\psi)
(dd^c \phi)^j
(dd^c \psi)^{n-j}.
\]

The functional $\phi \mapsto \mathcal E(\overline L_\phi)$
is concave and nondecreasing on the space of plurisubharmonic
weights, and is G\^ateaux differentiable. Moreover,
\[
\frac{d}{dt}
\Big(
\mathcal E(\overline L_{t\psi+(1-t)\phi})
-
\mathcal E(\overline L_\phi)
\Big)_{|_{t=0^+}}
=
\int_X (\phi-\psi)\, \mathrm{MA}(\phi).
\]

\medskip

The \emph{equilibrium weight} associated to $\phi$ is defined by
\begin{equation}\label{equiweight}
P_X\phi
=
\sup{}^\ast
\left\{
\vartheta \;\middle|\;
\vartheta \text{ psh weight on } L,\;
\vartheta \le \phi
\right\},
\end{equation}
where ${}^\ast$ denotes upper semicontinuous regularization
(see \cite{DemaillyLivre}).
The associated equilibrium measure is
\[
\mu_\phi := \mathrm{MA}(P_X\phi).
\]

\medskip

If $\phi$ is $\mathscr C^2$, Berman \cite{Berman2009} proved that
\[
\frac{1}{k^n}
\rho(\mu,k\phi)\,\mu
\longrightarrow
\mu_\phi
\quad \text{weakly as } k\to\infty,
\]
and that
\[
\mathrm{vol}(L)
=
n! \int_X \mu_\phi,
\]
where $
\mathrm{vol}(L)
=
\limsup_{k\to\infty}
\frac{\dim H^0(X,kL)}{k^n/n!}.$

\section{Normed $\mathbb{Z}$-modules}\label{SectionTheta}

A \emph{normed $\mathbb{Z}$-module} $\overline{E}=(E, \|\cdot\|)$ is a $\mathbb{Z}$-module of finite type $E$ endowed with a norm $\|\cdot\|$ on the $\mathbb{C}$-vector space $E_{\mathbb{C}} = E \otimes_{\mathbb{Z}} \mathbb{C}$. Let $E_{\mathrm{tors}}$ denote the torsion submodule of $E$, $E_{\mathrm{free}} = E / E_{\mathrm{tors}}$ its torsion-free part, and $E_{\mathbb{R}} = E \otimes_{\mathbb{Z}} \mathbb{R}$. We denote the unit ball in $E_{\mathbb{R}}$ by:
\[
B(\overline{E}) = \{m \in E_{\mathbb{R}} \mid \|m\| \leq 1\}.
\]
There exists a unique Haar measure on $E_{\mathbb{R}}$ such that $\mathrm{vol}(B(\overline{E})) = 1$. With respect to this measure, the \emph{arithmetic Euler characteristic} is defined as:
\[
\widehat{\chi}(\overline{E}) = \log \# E_{\mathrm{tors}} - \log \mathrm{vol}(E_{\mathbb{R}} / E_{\mathrm{free}}).
\]
Equivalently, for any choice of Haar measure on $E_{\mathbb{R}}$, we have:
\[
\widehat{\chi}(\overline{E}) = \log \# E_{\mathrm{tors}} - \log \left( \frac{\mathrm{vol}(E_{\mathbb{R}} / E_{\mathrm{free}})}{\mathrm{vol}(B(\overline{E}))} \right).
\]

The \emph{arithmetic degree} of $\overline{E}$ is defined as follows:
\[
\widehat{\deg}(\overline{E}) = \widehat{\chi}(\overline{E}) - \widehat{\chi}(\overline{\mathbb{Z}}^n),
\]
where $n = \mathrm{rank}(E_{\mathbb{R}})$ and $\widehat{\chi}(\overline{\mathbb{Z}}^n) = -\log \left( \Gamma(\frac{n}{2}+1)\pi^{-\frac{n}{2}} \right)$.

\medskip

\noindent \textsc{Theta Invariants and Successive Minima}

For $t > 0$, we define the theta function associated with $\overline{E}$ as:
\begin{equation}\label{defsmalltheta}
\theta_{\overline{E}}(t) = \sum_{v \in E} e^{-\pi t \|v\|^2}.
\end{equation}
We denote by $\widehat{H}^0(\overline{E})$ the set of small sections, by $\widehat{h}^0(\overline{E})$ its counting function, and by $h_\theta^0(\overline{E})$ the $\theta$-invariant:
\begin{equation}\label{defoftheta}
\widehat{H}^0(\overline{E}) = \{ m \in E \mid \|m\| \leq 1 \}, \quad \widehat{h}^0(\overline{E}) = \log \# \widehat{H}^0(\overline{E}), \quad h_\theta^0(\overline{E}) = \log \theta_{\overline{E}}(1).
\end{equation}
The higher-degree variants are defined as:
\[
\widehat{H}^1(\overline{E}) := \widehat{H}^0(\overline{E}^\vee), \quad \widehat{h}^1(\overline{E}) := \widehat{h}^0(\overline{E}^\vee), \quad h_\theta^1(\overline{E}) := h_\theta^0(\overline{E}^\vee),
\]
where $\overline{E}^\vee$ is the dual module $E^\vee = \mathrm{Hom}_{\mathbb{Z}}(E, \mathbb{Z})$ equipped with the dual norm:
\[
\|f\|^\vee = \sup_{m \in E_{\mathbb{R}} \setminus \{0\}} \frac{|f(m)|}{\|m\|}, \quad \forall f \in E^\vee.
\]
The \emph{first and $n$-th successive minima} are given respectively by:
\begin{equation}\label{deflambda1}
\lambda_1(\overline{E}) := \inf \{ \|v\| \mid v \in E \setminus \{0\} \},
\end{equation}
and
\begin{equation}\label{deflambdan}
\lambda_{n}(\overline{E}) := \inf \{ \lambda > 0 \mid \mathrm{Span}_{\mathbb{Q}} \left( B(E, \lambda^{-1}\|\cdot\|) \cap E \right) = E \}.
\end{equation}

\medskip

\noindent\textsc{Poisson Formula and Comparisons} 

When $\overline{E}$ is \emph{Euclidean}, the Poisson summation formula implies the following identity:
\begin{equation}\label{poisson}
\sum_{v \in E} e^{-\pi \|v\|^2} = \mathrm{covol}(\overline{E})^{-1} \sum_{v^\vee \in E^\vee} e^{-\pi \|v^\vee\|_{\overline{E}^{\vee}}^2}.
\end{equation}
In terms of the invariants defined above, this identity is expressed as:
\[
h_\theta^0(\overline{E}) - h_\theta^1(\overline{E}) = \widehat{\deg}(\overline{E}).
\]
Furthermore, $h_\theta^0(\overline{E})$ and $\widehat{h}^0(\overline{E})$ coincide up to an error term bounded only by the rank $n$ of $E$:
\begin{equation}\label{ArTheta}
h_\theta^0(\overline{E}) - \frac{n}{2} \log n + \log(1 - \tfrac{1}{2\pi}) \leq \widehat{h}^0(\overline{E}) \leq h_\theta^0(\overline{E}) + \pi,
\end{equation}
as shown in \cite[Theorem 3.1.1]{BostTheta}.

\begin{lemma}\label{lemma2} 
Let $\overline{E}$ be a Euclidean lattice of rank $n$. For all $t > 0$, we have:
\[
\sum_{v \in E} \|v\|^2 e^{-\pi t \|v\|^2} \leq \frac{n}{2\pi t} \sum_{v \in E} e^{-\pi t \|v\|^2}.
\]
\end{lemma}

\begin{proof}
By applying logarithmic differentiation with respect to $t$ to the Poisson summation formula, we obtain:
\begin{equation}\label{811}
2\pi \sum_{v \in E} t \|v\|^2 \frac{e^{-\pi t\|v\|^2}}{\sum_{u \in E} e^{-\pi t\|u\|^2}} = n - \frac{2\pi}{t} \sum_{v^\vee \in E^\vee} \|v^\vee\|_{\overline{E}^{\vee}}^2 \frac{e^{-\frac{\pi}{t}\|v^\vee\|_{\overline{E}^{\vee}}^2}}{\sum_{u^\vee \in E^\vee} e^{-\frac{\pi}{t}\|u^\vee\|_{\overline{E}^{\vee}}^2}},
\end{equation}
for all $t > 0$. Since the second term on the right-hand side is non-negative, it follows that:
\[
2\pi t \sum_{v \in E} \|v\|^2 e^{-\pi t\|v\|^2} \leq n \sum_{v \in E} e^{-\pi t\|v\|^2}.
\]
Dividing by $2\pi t$ yields the desired inequality.
\end{proof}

\section{A second-moment Gaussian inequality for Euclidean lattices
}

In this section we establish a second-moment Gaussian tail estimate for Euclidean lattices.
This inequality will serve as the analytic backbone of our proof of  Proposition \ref{vanishingTheta} and Theorems \ref{Thetafinitesum}, \ref{thm:enveloppe} and \ref{main}, allowing us to control the contribution of large-norm sections in the arithmetic volume asymptotics.

\medskip

A classical Gaussian tail bound for Euclidean lattices, due to Banaszczyk \cite{Bana}, plays a central role in transference theory.

\begin{proposition}\label{ineq1} 
Let $\overline{E}$ be a Euclidean lattice of rank $n$ and $\alpha > 0$. For every $r \geq \sqrt{\frac{n+2}{2\pi \alpha}}$, we have:
\begin{equation}\label{Bineq}
\frac{2\pi}{n} \frac{ \sum_{v \in E, \|v\| \geq r} \alpha \|v\|^2 e^{-\pi \alpha \|v\|^2} }{\sum_{v \in E} e^{-\pi \alpha \|v\|^2}} \leq \alpha^{1+n/2} e^{-\pi \alpha r^2} \frac{e^{1+n/2} }{ \left( \frac{n+2}{2\pi r^2} \right)^{1+n/2} }.
\end{equation}
In particular, for $\alpha = 1$ and $r = \tilde{r} \sqrt{\frac{n+2}{2\pi}}$ with $\tilde{r} > 1$, we obtain:
\begin{equation}\label{Bineq2}
\frac{2\pi}{n} \frac{ \sum_{v \in E, \|v\| \geq r} \|v\|^2 e^{-\pi \|v\|^2} }{\sum_{v \in E} e^{-\pi \|v\|^2}} \leq \left( \tilde{r}^2 e^{-(\tilde{r}^2-1)} \right)^{\frac{n+2}{2}}.
\end{equation}
In particular, for any fixed  $\tilde r>1$, the right-hand side decays exponentially in the rank  $n$.
\end{proposition}

\begin{proof}
Let $\alpha > 0$. For any $t \in (0, \alpha]$, we observe:
\[
\begin{aligned}
\sum_{v \in E, \|v\| \geq r} \|v\|^2 e^{-\pi \alpha \|v\|^2} &= \sum_{v \in E, \|v\| \geq r} \|v\|^2 e^{-\pi (\alpha - t) \|v\|^2} e^{-\pi t \|v\|^2} \\
&\leq e^{-\pi (\alpha - t)r^2} \sum_{v \in E} \|v\|^2 e^{-\pi t \|v\|^2}.
\end{aligned}
\]
Applying Lemma \ref{lemma2} to the right-hand side yields:
\[
\sum_{v \in E, \|v\| \geq r} \|v\|^2 e^{-\pi \alpha \|v\|^2} \leq e^{-\pi (\alpha - t)r^2} \frac{n}{2\pi t} \sum_{v \in E} e^{-\pi t \|v\|^2}.
\]
Recall that for a Euclidean lattice $\overline{E}$, the function $t \mapsto \log \theta_{\overline{E}}(t) + \frac{n}{2} \log t$ is increasing on $\mathbb{R}_{>0}$ (see \cite[Lemma 3.1.4]{BostTheta}). This monotonicity implies $\theta_{\overline{E}}(t) \leq ( \alpha / t )^{n/2} \theta_{\overline{E}}(\alpha)$ for $t \leq \alpha$. Substituting this into the inequality, we obtain:
\[
\frac{\sum_{v \in E, \|v\| \geq r} \alpha \|v\|^2 e^{-\pi \alpha \|v\|^2} }{\sum_{v \in E} e^{-\pi \alpha \|v\|^2}} \leq \frac{n \alpha^{1+n/2} e^{-\pi \alpha r^2}}{2\pi} \cdot \frac{e^{\pi r^2 t}}{t^{1+n/2}}.
\]
To minimize the right-hand side, we consider the function $f(t) = e^{\pi r^2 t} t^{-(1+n/2)}$. Its derivative vanishes at $t_0 = \frac{n+2}{2\pi r^2}$. By our hypothesis $r \geq \sqrt{\frac{n+2}{2\pi \alpha}}$, it follows that $t_0 \in (0, \alpha]$. Evaluating the expression at $t = t_0$ yields:
\[
\frac{n \alpha^{1+n/2} e^{-\pi \alpha r^2}}{2\pi} \cdot \frac{e^{1+n/2}}{\left( \frac{n+2}{2\pi r^2} \right)^{1+n/2}},
\]
which proves \eqref{Bineq}. The specialized form \eqref{Bineq2} follows by setting $\alpha=1$ and substituting the given value for $r$.
\end{proof}

\section{Arithmetic theta invariants 
 on arithmetic varieties}\label{sec4}

In this section we introduce theta-type invariants for Hermitian line bundles on arithmetic varieties and develop an arithmetic analogue of the variational framework established by Robert Berman \cite{Berman2009} for the complex distortion function.

\medskip

Let $\mathcal{X}$ be a smooth projective arithmetic variety over $\mathbb{Z}$ of dimension $n+1$. Let $X := \mathcal{X}(\mathbb{C})$ denote the associated complex manifold and $L := \mathcal{L}(\mathbb{C})$ the corresponding holomorphic line bundle. We fix a smooth, normalized volume form $\mu$ on $X$ (i.e. $\int_X\mu=1$) and let $\overline{\mathcal{L}}_\phi = (\mathcal{L}, \|\cdot\|_\phi)$ be a continuous Hermitian line bundle on $\mathcal{X}$. For any $k \in \mathbb{N}$, let $N_k$ denote the rank of the $\mathbb{Z}$-module $H^0(\mathcal{X}, k\mathcal{L})$.

For any $k \geq 1$, let $\overline{H^0(\mathcal{X}, k\mathcal{L})}_{(L^2, k\phi)}$ (resp. $\overline{H^0(\mathcal{X}, k\mathcal{L})}_{(\sup, k\phi)}$) denote the normed $\mathbb{Z}$-module $H^0(\mathcal{X}, k\mathcal{L})$ equipped with the $L^2$-norm $\|\cdot\|_{(L^2, k\phi)}$ (resp. the sup-norm $\|\cdot\|_{\sup, k\phi}$).

  \medskip
  
We define the first minima with respect to the $L^2$-norm and the supremum norm as follows:
\[
\lambda_1(\mu, k\phi) := \lambda_1\left(\overline{H^0(\mathcal{X}, k\mathcal{L})}_{(L^2, k\phi)}\right) \quad \text{and} \quad \lambda_1(\sup, k\phi) := \lambda_1\left(\overline{H^0(\mathcal{X}, k\mathcal{L})}_{(\sup, k\phi)}\right),
\]
where $\lambda_1(\cdot)$ is the invariant defined in \eqref{deflambda1}. 

The \emph{asymptotic first minimum} of $\overline{\mathcal{L}}_\phi$ is defined as:
\[
\widehat{\mu}_{\max}(\overline{\mathcal{L}}_\phi) := -\lim_{k \to \infty} \frac{1}{k} \log \lambda_1\left(\overline{H^0(\mathcal{X}, k\mathcal{L})}_{(\sup, k\phi)}\right).
\]
It is well known that this limit exists and is finite. By the Bernstein--Markov property, the limit remains invariant if the supremum norm is replaced by the $L^2$-norm.

\subsubsection*{Arithmetic Positivity}

We recall the following notions of arithmetic positivity for Hermitian line bundles:

\begin{definition}\label{def:arith-pos1}
The following notions of arithmetic positivity will be used:

\begin{enumerate}
\item (Arithmetic Ampleness)
A continuous Hermitian line bundle $\overline{\mathcal{L}}_\phi$ is said to be \emph{ample} if:
\begin{itemize}
    \item $\mathcal{L}$ is an ample line bundle on $\mathcal{X}$;
    \item For sufficiently large $k$, the $\mathbb{Z}$-module $H^0(\mathcal{X}, k\mathcal{L})$ is generated by its small sections:
    \[ \{s \in H^0(\mathcal{X}, k\mathcal{L}) \mid \|s\|_{\sup, k\phi} < 1\}. \]

\end{itemize}

\item {\rm{(Arithmetic Bigness)}}
A continuous Hermitian line bundle $\overline{\mathcal{L}}_\phi$ is said to be \emph{big} if:
\begin{itemize}
    \item $\mathcal{L}_{\mathbb{Q}}$ is big on $\mathcal{X}_{\mathbb{Q}}$;
    \item $\widehat{\mathrm{vol}}(\overline{\LL}_\phi)>0$. \end{itemize}

\item (Arithmetic Nefness)
A Hermitian line bundle $\overline{\mathcal{L}}_\phi$ is \emph{nef} if for every $1$-dimensional closed subscheme $\Gamma$ in $\mathcal{X}$, the arithmetic degree satisfies:
\[ \widehat{\deg}(\overline{\mathcal{L}}|_{\Gamma}) \geq 0. \]

\end{enumerate}

\end{definition}

\begin{definition}\label{def:weak-nef}
A Hermitian line bundle $\overline{\mathcal{L}}_\phi$ is said to be {weakly nef} if $\overline{\mathcal{L}}_\phi + \overline{\mathcal{A}}_\psi$ is ample for every ample Hermitian line bundle $\overline{\mathcal{A}}_\psi$.
\end{definition}

It follows immediately that \begin{center}
\textit{Ample} $\implies$ \textit{Weakly nef} $\implies$ \textit{Nef}.
\end{center}
The first implication follows directly from the definitions, while the second is immediate from the positivity of arithmetic degrees.

\begin{proposition}\label{wnef}
Let $\overline{\mathcal{L}}_\phi$ be a nef Hermitian line bundle in the sense of Zhang \cite{Zha95}. Then $\overline{\mathcal{L}}_\phi$ is weakly nef.
\end{proposition}

\begin{proof}
Let $\overline{\mathcal{A}}_\psi$ be an ample Hermitian line bundle. We first recall that the equilibrium weight $P_{\mathcal{X}}\psi$, defined as in \eqref{equiweight}, is continuous. By the definition of ampleness, there exists a sufficiently large integer $k$ such that the $\mathbb{Z}$-module $H^0(\mathcal{X}, k\mathcal{A})$ is generated by the set $\{s \in H^0(\mathcal{X}, k\mathcal{A}) \mid \|s\|_{\sup, k\psi} < 1\}$. Since $\|s\|_{\sup, k\psi} = \|s\|_{\sup, kP_{\mathcal{X}}\psi}$, it follows that $\overline{\mathcal{A}}_{P_{\mathcal{X}}\psi}$ is also ample.

Using Richberg's theorem and the regularization techniques of Demailly \cite{DemaillyLivre}, we can construct a sequence $(\psi_j)_{j \in \mathbb{N}}$ of smooth, strictly plurisubharmonic weights on $\mathcal{A}$ such that $\|\psi_j - P_{\mathcal{X}}\psi\|_{\sup} \to 0$ as $j \to \infty$. 

Let $\alpha = \max \{ \|s\|_{\sup, k\psi} \mid s \in \widehat{H}^0(\overline{H^0(\mathcal{X}, k\mathcal{A})}_{\sup, k\psi}) \}$. For any $\epsilon \in (0, -\frac{1}{k}\log \alpha)$, there exists $N \in \mathbb{N}$ such that for all $j \geq N$ and all $s$ in the generating set, we have $\|s\|_{\sup, k\psi_j} \leq e^{k\epsilon} \alpha < 1$. Thus, $\overline{\mathcal{A}}_{\psi_j}$ is ample in the sense of Zhang for $j \geq N$. By the \emph{arithmetic Nakai--Moishezon Theorem} \cite{Zha95}, the sum $\overline{\mathcal{L}}_\phi + \overline{\mathcal{A}}_{\psi_j}$ is ample for all $j \geq N$.

To conclude, we compare the successive minima. For every $j \in \mathbb{N}$, we have the following Lipschitz-type estimate:

\[
\begin{aligned}
\Biggl| \limsup_{\ell \to \infty} & \frac{\log \lambda_{N_\ell}(\overline{H^0(\mathcal{X}, \ell(\mathcal{L}+\mathcal{A}))}_{\sup, \ell(\phi+\psi)})}{\ell} \\
&- \limsup_{\ell \to \infty} \frac{\log \lambda_{N_\ell}(\overline{H^0(\mathcal{X}, \ell(\mathcal{L}+\mathcal{A}))}_{\sup, \ell(\phi+\psi_j)})}{\ell} \Biggr| \leq \|\psi_j - P_{\mathcal{X}}\psi\|_{\sup},
\end{aligned}
\]

where $N_\ell = \mathrm{rank}\, H^0(\mathcal{X}, \ell(\mathcal{L}+\mathcal{A}))$. As $j \to \infty$, the right side vanishes. Since $\overline{\mathcal{L}}_\phi + \overline{\mathcal{A}}_{\psi_j}$ is ample, its $N_\ell$-th successive minimum is strictly less than $1$ for $\ell \gg 0$, implying:
\[
\limsup_{\ell \to \infty} \frac{1}{\ell} \log \lambda_{N_\ell}(\overline{H^0(\mathcal{X}, \ell(\mathcal{L}+\mathcal{A}))}_{\sup, \ell(\phi+\psi)}) \leq 0.
\]
By choosing a sufficiently small $\epsilon > 0$ such that $\overline{\mathcal{A}}_{\psi-\epsilon}$ remains ample, the same argument shows the limsup is strictly negative, which establishes that $\overline{\mathcal{L}}_\phi + \overline{\mathcal{A}}_\psi$ is generated by small sections and is thus ample.
\end{proof}

\begin{definition}\label{defTheta}
Let $E \subseteq H^0(\mathcal{X}, \mathcal{L})$ be a non-zero $\mathbb{Z}$-submodule. For every $x \in \mathcal{X}(\mathbb{C})$, we define the \emph{arithmetic distortion function} associated with $(\mu, \phi)$ as:
\begin{equation}\label{ThetaDef}
\Theta(\mu, \overline{E})(x) := 2\pi \sum_{v \in E} \|v(x)\|_\phi^2 \frac{e^{-\pi \|v\|_{(L^2, \phi)}^2}}{\sum_{u \in E} e^{-\pi \|u\|_{(L^2, \phi)}^2}},
\end{equation}
and we denote 
\begin{equation}\label{BergmanDef}
\rho(\mu, \overline{E})(x) := \sup_{s \in E_{\mathbb{C}} \setminus \{0\}} \frac{\|s(x)\|_\phi^2}{\quad \|s\|_{(L^2, \phi)}^2},
\end{equation}
where $\overline{E}$ is endowed with the $L^2$-Hermitian structure induced by $(\mu, \phi)$. When $E = H^0(\mathcal{X}, \mathcal{L})$, we simplify the notation to $\Theta(\mu, \phi)(x)$ or $\Theta(\mu; \overline{\mathcal{L}}_\phi)(x)$.
\end{definition}

The function $\Theta(\mu,\overline E)$ can be viewed as the expectation of the pointwise norm with respect to the Gaussian probability measure induced by the $L^2$-norm on $E$.

\subsection{On the arithmetic theta functions, Part I}
In this section, we analyze the relationship between different formulations of the arithmetic volume, specifically focusing on the $\theta$-variants $h_\theta^0$. The following results ensure that our definition of $\widehat{\mathrm{vol}}(\overline{\mathcal{L}}_\phi)$ is robust with respect to the choice of $L^2$ or supremum norms.

\medskip

We maintain the notation established in Section \ref{sec4}. Recall the following definitions of the arithmetic volume functions:
\begin{equation}\label{defvol}
\begin{aligned}
\widehat{\mathrm{vol}}(\overline{{\mathcal{L}}}_\phi) &:= \limsup_{k \to \infty} \frac{\widehat{h}^0\left(\overline{H^0(\mathcal{X}, k\mathcal{L})}_{(\sup, k\phi)}\right)}{k^{n+1}/(n+1)!}, \\
\widehat{\mathrm{vol}}_{L^2}(\overline{{\mathcal{L}}}_\phi) &:= \limsup_{k \to \infty} \frac{\widehat{h}^0\left(\overline{H^0(\mathcal{X}, k\mathcal{L})}_{(L^2, k\phi)}\right)}{k^{n+1}/(n+1)!}.
\end{aligned}
\end{equation}

In analogy with the above, we introduce the following $\theta$-invariants for the arithmetic volume:
\begin{equation}
\begin{aligned}
\widehat{\mathrm{vol}}_{\sup, \theta}(\overline{{\mathcal{L}}}_\phi) &:= \limsup_{k \to \infty} \frac{h_\theta^0\left(\overline{H^0(\mathcal{X}, k\mathcal{L})}_{\sup, k\phi}\right)}{k^{n+1}/(n+1)!}, \\
\widehat{\mathrm{vol}}_{L^2, \theta}(\overline{{\mathcal{L}}}_\phi) &:= \limsup_{k \to \infty} \frac{h_\theta^0\left(\overline{H^0(\mathcal{X}, k\mathcal{L})}_{L^2, k\phi}\right)}{k^{n+1}/(n+1)!}.
\end{aligned}
\end{equation}

The following lemma describes the variation of the $h_\theta^0$ invariant with respect to the weight, providing an arithmetic analogue of the variation of the Bergman  function.

\begin{lemma}\label{difference} 
Let $\psi$ and $\phi$ be two continuous weights on $\mathcal{L}$. For all $k \in \mathbb{N}$, we have:
\[
\begin{split}
h_\theta^0\left(\overline{H^0(\mathcal{X}, k\mathcal{L})}_{(L^2, k\phi)}\right) - h_\theta^0&\left(\overline{H^0(\mathcal{X}, k\mathcal{L})}_{(\mu, k\psi)}\right) \\
&= k \int_X (\phi - \psi) \left( \int_0^1 \Theta(\mu, k\phi_t)(x) \, dt \right) {{\mu}},
\end{split}
\]
where $\phi_t = t\phi + (1-t)\psi$ for $t \in [0,1]$.
\end{lemma}

\begin{proof}
See \cite[Proposition 4.6]{HajliTAMS}.
\end{proof}

\begin{proposition}
Let $\mu$ be a normalized smooth volume form on $\mathcal{X}(\mathbb{C})$ and $\overline{\mathcal{L}}_\phi$ a continuous Hermitian line bundle on $\mathcal{X}$. Then:
\begin{equation}\label{volthetal2sup}
\widehat{\mathrm{vol}}_{\sup, \theta}(\overline{{\mathcal{L}}}_\phi) = \widehat{\mathrm{vol}}_{L^2, \theta}(\overline{{\mathcal{L}}}_\phi).
\end{equation}
\end{proposition}

\begin{proof}
We compare the invariants $h_\theta^0$ under the supremum and $L^2$ norms. Since $\mu$ is a smooth volume form and $\phi$ is continuous, the Bernstein--Markov property implies that for every $\varepsilon > 0$, there exists a constant $C > 0$ such that for sufficiently large $k$:
\[
\|\cdot\|_{(L^2, k\phi)} \leq \|\cdot\|_{\sup, k\phi} \leq C e^{k\varepsilon} \|\cdot\|_{(L^2, k\phi)}.
\]
This comparison leads to the following inequalities:
\[
h_\theta^0\left(\overline{H^0(\mathcal{X}, k\mathcal{L})}_{(L^2, k\psi)}\right) \leq h_\theta^0\left(\overline{H^0(\mathcal{X}, k\mathcal{L})}_{\sup, k\phi}\right) \leq h_\theta^0\left(\overline{H^0(\mathcal{X}, k\mathcal{L})}_{(L^2, k\phi)}\right),
\]
where we set $\psi = \phi + \varepsilon - \frac{1}{k}\log C$. Applying Lemma \ref{difference} and observing that $\Theta(\mu, k\phi) \leq \rho(\mu, k\phi)$, we evaluate the difference:
\[
h_\theta^0\left(\overline{H^0(\mathcal{X}, k\mathcal{L})}_{(\mu, k\psi)}\right) = h_\theta^0\left(\overline{H^0(\mathcal{X}, k\mathcal{L})}_{(\mu, k\phi)}\right) + O\left(\max(\varepsilon, k^{-1})\right) k^{n+1}.
\]
Dividing the expressions by $k^{n+1}/(n+1)!$ and taking the limit as $k \to \infty$ followed by $\varepsilon \to 0$, we obtain the identity \eqref{volthetal2sup}.
\end{proof}


\begin{theorem}\label{MorMor}
Let \(\mu\) be a normalized smooth volume form on \(\mathcal{X}(\mathbb{C})\), and let \(\overline{\mathcal{L}}_\phi\) be a continuous Hermitian line bundle on \(\mathcal{X}\). Then:
\[
\widehat{\mathrm{vol}}_{L^2}(\overline{\mathcal{L}}_\phi)
= \widehat{\mathrm{vol}}_{L^2, \theta}(\overline{\mathcal{L}}_\phi)
= \widehat{\mathrm{vol}}(\overline{\mathcal{L}}_\phi)
= \widehat{\mathrm{vol}}_{\sup, \theta}(\overline{\mathcal{L}}_\phi).
\]
\end{theorem}

Theorem \ref{MorMor} shows that the arithmetic volume can be entirely recovered from the asymptotics of the  theta invariants $h_\theta^0$. 

\begin{proof}[Proof of Theorem \ref{MorMor}]
From \eqref{ArTheta}, it follows directly that:
\begin{equation}\label{cont}
\widehat{\mathrm{vol}}_{L^2, \theta}(\overline{\mathcal{L}}_\phi) = \widehat{\mathrm{vol}}_{L^2}(\overline{\mathcal{L}}_\phi).
\end{equation}

Let \(\epsilon \in \mathbb{R}\). We apply Lemma \ref{difference} to obtain the following identity:
\[
h_\theta^0\bigl(\overline{H^0(\mathcal{X}, k\mathcal{L})}_{(L^2, k\phi)}\bigr) - h_\theta^0\bigl(\overline{H^0(\mathcal{X}, k\mathcal{L})}_{(\mu, k(\phi-\epsilon))}\bigr) = \epsilon k \int_X \int_0^1 \Theta(x, k\phi_t) \, {{\mu}} \, dt, \quad \forall k \in \mathbb{N},
\]
where \(\phi_t = t\phi + (1-t)(\phi-\epsilon) = \phi - (1-t)\epsilon\) for \(t \in [0,1]\). Utilizing the bound \(\Theta(\mu, \phi) \leq \rho(\mu, \phi)\) and \eqref{intRho}, we deduce:
\[
\Bigl| h_\theta^0\bigl(\overline{H^0(\mathcal{X}, k\mathcal{L})}_{(L^2, k\phi)}\bigr) - h_\theta^0\bigl(\overline{H^0(\mathcal{X}, k\mathcal{L})}_{(\mu, k(\phi-\epsilon))}\bigr) \Bigr| \leq \epsilon k N_k.
\]
Dividing by \(k^{n+1}/(n+1)!\) and taking the limit as \(k \to \infty\), it follows that:
\[
\widehat{\mathrm{vol}}_{L^2, \theta}(\overline{\mathcal{L}}_\phi) - \widehat{\mathrm{vol}}_{L^2, \theta}(\overline{\mathcal{L}}_{\phi-\epsilon}) = O(\epsilon).
\]
In view of \eqref{cont}, the preceding estimate implies:
\[
\widehat{\mathrm{vol}}_{L^2}(\overline{\mathcal{L}}_\phi) - \widehat{\mathrm{vol}}_{L^2}(\overline{\mathcal{L}}_{\phi-\epsilon}) = O(\epsilon).
\]

Since \(\mu\) is smooth and the metrics are continuous, there exists a constant \(C > 0\) independent of \(k\) such that for sufficiently large \(k\):
\[
\|s\|_{\sup, k\phi} \leq C e^{k\epsilon} \|s\|_{\mu, k\phi} = \|s\|_{\mu, k\psi}, \quad \forall s \in H^0(\mathcal{X}, k\mathcal{L}),
\]
where we have set  \(\psi := \phi + \epsilon - \frac{1}{k}\log C\). Thus, we have:
\begin{equation*}
\begin{aligned}
\widehat{h}^0\bigl(\overline{H^0(\mathcal{X}, k\mathcal{L})}_{(\sup, k\phi)}\bigr) &\geq \widehat{h}^0\bigl(\overline{H^0(\mathcal{X}, k\mathcal{L})}_{(L^2, k\psi)}\bigr) \\
&\geq h_\theta^0\bigl(\overline{H^0(\mathcal{X}, k\mathcal{L})}_{(L^2, k\psi)}\bigr) + O(k^n \log k) \quad (\text{by } \eqref{ArTheta}) \\
&= h_\theta^0\bigl(\overline{H^0(\mathcal{X}, k\mathcal{L})}_{(L^2, k\phi)}\bigr) + O\bigl(\max(\epsilon, k^{-1})\bigr)k^{n+1} + O(k^n \log k).
\end{aligned}
\end{equation*}
This yields:
\[
\widehat{\mathrm{vol}}(\overline{\mathcal{L}}_\phi) \geq \widehat{\mathrm{vol}}_{L^2, \theta}(\overline{\mathcal{L}}_\phi) + O(\epsilon),
\]
which implies \(\widehat{\mathrm{vol}}(\overline{\mathcal{L}}_\phi) \geq \widehat{\mathrm{vol}}_{L^2, \theta}(\overline{\mathcal{L}}_\phi)\) by letting \(\epsilon \to 0\). Conversely, by a symmetric argument using the standard comparison of norms, we show:
\[
\widehat{\mathrm{vol}}_{L^2}(\overline{\mathcal{L}}_\phi) - \widehat{\mathrm{vol}}(\overline{\mathcal{L}}_\phi) = O(\epsilon).
\]
Taking \(\epsilon \to 0\), we conclude:
\[
\widehat{\mathrm{vol}}(\overline{\mathcal{L}}_\phi) = \widehat{\mathrm{vol}}_{L^2}(\overline{\mathcal{L}}_\phi).
\]
In particular, we obtain the continuity property:
\begin{equation}\label{continuitytrivial}
\widehat{\mathrm{vol}}(\overline{\mathcal{L}}_\phi) - \widehat{\mathrm{vol}}(\overline{\mathcal{L}}_{\phi-\epsilon}) = O(\epsilon).
\end{equation}
\end{proof}


\subsection{On the arithmetic theta functions, Part II}

We maintain the notation introduced in Section~\ref{sec4}. 
This subsection is devoted to the study of the asymptotic expansion of 
$\Theta(\mu, k\phi)$ and to the arithmetic measures naturally arising from it.

\medskip

Let $t>0$. Consider the $\mathbb{Z}$-submodule generated by the small sections
\[
E_{k,t}
:=
\left\langle 
\widehat{H}^0\!\left(
\overline{H^0(\mathcal{X}, k\mathcal{L})}_{(L^2, k(\phi-t))}
\right)
\right\rangle
=
\left\langle
\left\{
s \in H^0(\mathcal{X}, k\mathcal{L})
\;\middle|\;
\|s\|_{L^2,k\phi} \le e^{-kt}
\right\}
\right\rangle .
\]
In other words, $E_{k,t}$ is the $\mathbb{Z}$-span of the set of small sections
$\widehat{H}^0(\mathcal{X}, k\mathcal{L})_{L^2, k(\phi-t)}$. 
We denote by $\overline{E}_{k,t}$ the Euclidean lattice $E_{k,t}$ endowed with the norm induced from 
$\overline{H^0(\mathcal{X}, k\mathcal{L})}_{(L^2, k(\phi-t))}$.

It follows immediately from the definition of the distortion function that
\begin{equation}\label{y1}
\int_{\mathcal{X}(\mathbb{C})}
\rho(\mu, \overline{E}_{k,t})(x)
\, \mu
=
\dim_{\mathbb{C}} \bigl( (E_{k,t})_{\mathbb{C}} \bigr).
\end{equation}

By \cite[Proposition~3.4]{YuanZhang2}, we therefore obtain
\[
\limsup_{k \to \infty}
\int_{\mathcal{X}(\mathbb{C})}
\frac{\rho(\mu, \overline{E}_{k,t})(x)}{k^n/n!}
\, \mu
=
\mathrm{dvol}(\overline{\mathcal{L}}_t).
\]

Using Proposition~\ref{ineq1}  to compare  $\Theta$ and $\rho$ established, one deduces that
\[
\limsup_{k \to \infty}
\int_{\mathcal{X}(\mathbb{C})}
\frac{\Theta(\mu, \overline{E}_{k,t})(x)}{k^n/n!}
\, \mu
=
\mathrm{dvol}(\overline{\mathcal{L}}_t).
\]

\medskip

These asymptotic identities naturally lead to the following definitions.

\begin{definition}
Let $\mathcal{X}$ be a smooth arithmetic variety over $\mathrm{Spec}(\mathbb{Z})$. 
Let $\overline{\mathcal{L}}_\phi$ be a continuous Hermitian line bundle on $\mathcal{X}$, and let $\mu$ be a smooth volume form on $\mathcal{X}(\mathbb{C})$.

Let $\overline{E}_\bullet = \bigoplus_{k \in \mathbb{N}} \overline{E}_k$, 
where each $\overline{E}_k = (E_k, \|\cdot\|_{L^2,k\phi})$ and 
$E_k \subset H^0(\mathcal{X}, k\mathcal{L})$ is a normed $\mathbb{Z}$-submodule (see Section~\ref{SectionTheta}). 

We define the upper arithmetic equilibrium measures by
\[
\widehat{\mu}^{\sup}_{\mathrm{eq}}((\mathcal{X},\mu), \overline{E}_\bullet)
:=
\limsup_{k \to \infty}
\frac{1}{k^n}
\Theta(\mu, \overline{E}_k)\, \mu,
\]
and
\[
\widehat{\omega}^{\sup}_{\mathrm{eq}}((\mathcal{X},\mu), \overline{E}_\bullet)
:=
\limsup_{k \to \infty}
\frac{1}{k^n}
\rho(\mu, \overline{E}_k)\, \mu.
\]

When
\[
\overline{E}_\bullet
=
\bigoplus_{k \in \mathbb{N}}
\overline{H^0(\mathcal{X}, k\mathcal{L})}_{L^2, k\phi},
\]
we simply write
\[
\widehat{\mu}^{\sup}_{\mathrm{eq}}((\mathcal{X},\mu), \overline{\mathcal{L}}_\phi)
\quad \text{and} \quad
\widehat{\omega}^{\sup}_{\mathrm{eq}}((\mathcal{X},\mu), \overline{\mathcal{L}}_\phi).
\]

The lower measures
$\widehat{\mu}^{\inf}_{\mathrm{eq}}(\cdot)$ and 
$\widehat{\omega}^{\inf}_{\mathrm{eq}}(\cdot)$
are defined analogously by replacing $\limsup$ with $\liminf$. 
\end{definition}

The invariant $\widehat{\mu}^{\sup}_{\mathrm{eq}}((\mathcal{X},\mu), \overline{\mathcal{L}}_\phi)$ provides a refined arithmetic perspective compared to the mere set of small sections.  Indeed, while the arithmetic volume $\widehat{\mathrm{vol}}(\overline{\mathcal{L}}_\phi)$ captures the global arithmetic magnitude of a Hermitian line bundle, $\widehat{\mu}^{\sup}_{\mathrm{eq}}((\mathcal{X},\mu), \overline{\mathcal{L}}_\phi)$ is expected to reflect finer local arithmetic positivity properties. In this sense, it serves as an arithmetic analogue of the equilibrium measure, characterizing the arithmetic positivity of Hermitian line bundles.

\medskip

The following proposition establishes an arithmetic counterpart to the classical identity for the Bergman kernel. It demonstrates that as the scaling parameter $s$ tends toward negative infinity, the integral of the arithmetic distortion function recovers the rank of the section module.
\begin{proposition} 
Let $\overline{\mathcal{L}}_\phi$ be a continuous Hermitian line bundle on $\mathcal{X}$, and let $\mu$ be a smooth volume form on $\mathcal{X}(\mathbb{C})$. For every $k \in \mathbb{N}$, we have:
\[
\lim_{s\rightarrow -\infty}\int_X \Theta(\mu,k(\phi-s))(x) {{\mu}} = \dim H^0(\mathcal{X}_\mathbb{Q}, k\mathcal{L}_\mathbb{Q}).
\]
\end{proposition}

\begin{proof}
Let $\overline{E}$ be a Euclidean lattice. From the Poisson summation identity \eqref{poisson}, we infer the following relation:
\begin{equation}\label{811}
2\pi \sum_{v\in E} t \|v\|^2_{\overline E} \frac{e^{-\pi t\|v\|^2}}{\sum_{u\in E}  e^{-\pi t\|u\|_{\overline E}^2 }} =
\mathrm{rank}(E) - 
\frac{2\pi}{t} \sum_{v^\vee\in E^\vee} \|v^\vee\|^2_{\overline{E}^\vee} \frac{e^{-\frac{\pi}{ t}\|v^\vee\|^2}}{\sum_{u \in E^\vee}  e^{-\frac{\pi}{ t}\|u^{\vee}\|_{\overline{E}^\vee}^2 }},\quad \forall t>0.
\end{equation}

Let $\lambda = \lambda_1(\overline{E}^\vee)$ be the first minimum of the dual lattice, and let $\nu$ denote the number of elements in $E^\vee$ with norm equal to $\lambda$. For every $0 < t < 1$, we observe:
\[
\begin{aligned}
\frac{1}{t}\sum_{v^\vee\in E^\vee} \|v^\vee\|^2_{\overline{E}^\vee} \frac{e^{-\frac{\pi}{ t}\|v^\vee\|^2}}{\sum_{u\in E^\vee}  e^{-\frac{\pi}{ t}\|u^{\vee}\|_{\overline{E}^\vee}^2 }} &= \frac{\nu}{t} \lambda^2 e^{-\frac{\pi}{t}\lambda^2} + \frac{1}{t}\sum_{\|v^\vee\|>\lambda} \|v^\vee\|^2_{\overline{E}^\vee} \frac{e^{-\frac{\pi}{ t}\|v^\vee\|^2}}{\sum_{u\in E^\vee}  e^{-\frac{\pi}{ t}\|u^{\vee}\|_{\overline{E}^\vee}^2 }} \\ 
&\leq \frac{\nu}{t} \lambda^2 e^{-\frac{\pi}{t}\lambda^2} + \frac{1}{t} \sum_{\|v^\vee\|>\lambda} \|v^\vee\|^2_{\overline{E}^\vee} e^{-\pi \|v^\vee\|^2} e^{-\pi \lambda^2 (\frac{1}{t}-1)}.
\end{aligned}
\]
From this, we conclude that:
\begin{equation}\label{812}
\frac{1}{t}\sum_{v^\vee\in E^\vee} \|v^\vee\|^2_{\overline{E}^\vee} \frac{e^{-\frac{\pi}{ t}\|v^\vee\|^2}}{\sum_{u\in E^\vee}  e^{-\frac{\pi}{ t}\|u^{\vee}\|_{\overline{E}^\vee}^2 }} = \frac{\nu}{t} \lambda^2 e^{-\frac{\pi}{t}\lambda^2} + \frac{1}{t} e^{-\frac{\pi \lambda^2}{t}} O(1).
\end{equation}
Note that the integral of the arithmetic distortion function is given by:
\[
\int_X \Theta(\mu,k(\phi-s)) {{\mu}} = 2\pi \sum_{v\in H^0(\mathcal{X},k\mathcal{L})} e^{2ks} \|v\|_{(L^2,k\phi)}^2    \frac{e^{-\pi e^{2ks}\|v\|_{(L^2,k\phi)}^2}}{\sum_{u\in H^0(\mathcal{X},k\mathcal{L})} e^{-\pi e^{2ks}\|u\|_{(L^2,k\phi)}^2}}.
\]
Setting $t = e^{2ks}$, we see that as $s \to -\infty$, $t \to 0^+$. In this limit, the term in \eqref{812} vanishes exponentially. Thus, by substituting \eqref{812} into \eqref{811}, the proposition follows.
\end{proof}

\begin{proposition}\label{vanishingTheta} 
Let $\overline{\mathcal{L}}_\phi$ be a continuous Hermitian line bundle on $\mathcal{X}$. Let $\mu$ be a smooth volume form on $\mathcal{X}(\mathbb{C})$.
Let $t>0$ such that $t > \widehat{\mu}_{\max}(\overline{\mathcal{L}}_\phi)$. Then we have:
\[
\lim_{k\rightarrow \infty}\frac{1}{k^n} \Theta(\mu,k(\phi-t))(x) = 0 \quad \forall x\in \mathcal{X}(\mathbb{C}).
\]
In other words,
\[
\widehat{\mu}^{\sup}_{\mathrm{eq}}((\mathcal{X},\mu), \overline{\mathcal{L}}_\phi)=0.
\]
while $\widehat{\omega}^{\sup}_{\mathrm{eq}}((\mathcal{X},\mu), \overline{\mathcal{L}}_\phi)\neq 0$ when the metric is positive.
\end{proposition}

\begin{proof}
Since $t > \widehat{\mu}_{\max}(\overline{\mathcal{L}}_\phi)$, it follows from the definition of the maximal slope $\widehat{\mu}_{\max}(\overline{\mathcal{L}}_\phi) = -\lim_{k\rightarrow \infty} \frac{1}{k} \log \lambda_1(\mu, k\phi)$ that:
\[
\lambda_1(\mu, k\phi) \geq \sqrt{\frac{ 2+N_k}{2\pi e^{2kt}}}, \quad \text{for } k \gg 1.
\]
By setting $\alpha = e^{2kt}$ and $r = \lambda_1(\mu, k\phi)$ in the inequality \eqref{Bineq}, and denoting $N_k = \mathrm{rank}(E)$, we obtain:
\[
\alpha^{1+\frac{N_k}{2}} e^{-\pi \alpha r^2} \frac{e^{1+\frac{N_k}{2} } }{\left( \frac{2+N_k}{2\pi r^2 }\right)^{1+\frac{N_k}{2}} } = e^{kt(N_k+2)} e^{-\pi \left( e^t \lambda_1(\mu, k\phi)^{1/k} \right)^{2k} } \frac{e^{\frac{2+N_k}{2}} }{ \left( \frac{2+N_k}{2\pi \lambda_1(\mu, k\phi)^2 }\right)^{\frac{2+N_k}{2}}}.
\]
As $k \rightarrow \infty$, and using asymptotic behavior of $N_k=O(k^n)$, the asymptotic expansion of the exponent yields:
\[
\begin{aligned}
(2kt + 1) \frac{N_k+2}{2} &- \pi \left(e^t \lambda_1(\mu, k\phi)^{1/k}\right)^{2k} - \frac{N_k+2}{2} \log \left( \frac{2+N_k}{2\pi \lambda_1(\mu, k\phi)^2 }\right) \\
&= -\pi \left(e^{t -\widehat{\mu}_{\max}(\overline{\mathcal{L}}_\phi)}\right)^{2k} + O(k^{n+1}).
\end{aligned}
\]
Since $t - \widehat{\mu}_{\max}(\overline{\mathcal{L}}_\phi) > 0$, the expression is dominated by the term $-\pi (e^{t - \widehat{\mu}_{\max}})^{2k}$, which ensures exponential decay. Finally, applying the Bernstein-Markov property for the smooth volume form $\mu$ and the continuous metric $\phi$, we have the pointwise estimate:
\[
\|s(x)\| \leq C e^{kt} \|s\|_{(L^2, k\phi)} \quad \forall k \gg 1,
\]
where the constant 
$C$ is uniform in $k$,
which concludes the proof of the vanishing of the distortion function.\\

In the case when the metric is positive, then we know that
\[
\widehat{\omega}^{\sup}_{\mathrm{eq}}((\mathcal{X},\mu), \overline{\mathcal{L}}_\phi)=\limsup_k \rho(\mu,k\phi)\mu= c_1(\overline \LL_\phi)^n.
\]
This concludes the proof of the proposition.
\end{proof}

\medskip

We know that
\[
\Theta(\mu,k\phi)\leq \rho(\mu,k\phi).
\]
We now analyze how sharp this inequality is at the level of asymptotics. 
To this end, we introduce a natural arithmetic graded object 
$\overline{S}_{\phi,\bullet}$ consisting of the submodules generated by small sections of $\overline{\mathcal{L}}_\phi$. 

More precisely, for each $k\geq 1$ we set
\[
S_{\phi,k}
:=
\left\langle 
\left\{ 
s \in H^0(\mathcal{X},k\mathcal{L})
\;\middle|\;
\|s\|_{\sup,k\phi} < 1
\right\}
\right\rangle .
\]
We denote by $\overline{S}_{\phi,k}$ the $\mathbb{Z}$-module $S_{\phi,k}$ endowed with the norm induced from 
$\overline{H^0(\mathcal{X},k\mathcal{L})}_{(L^2,k\phi)}$.

The following refined inequality clarifies further the role of the arithmetic equilibrium invariant.

\begin{theorem}\label{thm:enveloppe}
With the above notation, for every $t>0$ we have
\begin{equation}\label{eq:sharp-enveloppe}
\widehat{\omega}^{\sup}_{\mathrm{eq}}\big((\mathcal{X},\mu), \overline{S}_{\phi,\bullet}\big)
\;\leq\;
\widehat{\mu}^{\sup}_{\mathrm{eq}}\big((\mathcal{X},\mu), \overline{\mathcal{L}}_\phi\big)
\;\leq\;
\widehat{\omega}^{\sup}_{\mathrm{eq}}\big((\mathcal{X},\mu), \overline{S}_{\phi+t,\bullet}\big),
\end{equation}
and
\[
\widehat{\omega}^{\inf}_{\mathrm{eq}}\left((\mathcal X,\mu),\overline{S}_{\phi,\bullet}\right)
\le
\widehat{\mu}^{\inf}_{\mathrm{eq}}\left((\mathcal X,\mu),\overline{\mathcal L}_\phi\right)
\le
\widehat{\omega}^{\inf}_{\mathrm{eq}}\left((\mathcal X,\mu),\overline{S}_{\phi+t,\bullet}\right).
\]

\end{theorem}

Observe that 
\[
t \longmapsto 
\widehat{\omega}^{\sup}_{\mathrm{eq}}\big((\mathcal{X},\mu), \overline{S}_{\phi+t,\bullet}\big)
\]
is nonincreasing. Moreover, for every $t>0$ we have
\[
\widehat{\omega}^{\sup}_{\mathrm{eq}}\big((\mathcal{X},\mu), \overline{S}_{\phi+t,\bullet}\big)
\leq
\widehat{\omega}^{\sup}_{\mathrm{eq}}\big((\mathcal{X},\mu), \overline{\mathcal{L}}_\phi\big).
\]
Hence inequality~\eqref{eq:sharp-enveloppe} exhibits 
\[
\widehat{\mu}^{\sup}_{\mathrm{eq}}\big((\mathcal{X},\mu), \overline{\mathcal{L}}_\phi\big)
\]
as the natural envelope between two arithmetic equilibrium measures generated by strictly small sections.

In particular, this shows that the invariant 
$\widehat{\mu}^{\sup}_{\mathrm{eq}}((\mathcal{X},\mu), \overline{\mathcal{L}}_\phi)$
is intrinsically arithmetic in nature and plays a canonical role in the theory of arithmetic volume functions.

\medskip

We shall prove that this theorem is a consequence of the following.

\begin{theorem}\label{Thetafinitesum} 
We have the following equality of measures on $X$:
\[
\limsup_{k \to \infty} \frac{\Theta(\mu, k\phi)}{k^n} \, \mu = \limsup_{k \to \infty} \frac{\rho(\mu, \overline{E}_{\phi,k})}{k^n} \, \mu,
\]
where \( E_{k} \) is the \(\mathbb{Z}\)-submodule of \( H^0(\mathcal{X}, k\mathcal{L}) \) generated by sections whose sup-norm satisfies \(\|v\|_{\sup, k\phi} \leq k^{2n}\). In other words,
\[
\widehat{\mu}^{\sup}_{\mathrm{eq}}((\mathcal{X},\mu), \overline{\mathcal{L}}_\phi) =\widehat{\omega}^{\sup}_{\mathrm{eq}}((\mathcal{X},\mu), \overline{E}_{\phi,\bullet})
\]
\end{theorem}

\begin{proof}
Define the ratio of the mass of sections with large sup-norm to the total mass:
\[
f(k) = \frac{\sum_{\|u\|_{\sup, k\phi} \geq k^{2n}} e^{-\pi \|u\|_{L^2, k\phi}^2}}{\sum_{u \in H^0(\mathcal{X}, k\mathcal{L})} e^{-\pi \|u\|_{L^2, k\phi}^2}}
\]
for \(k \in \mathbb{N}\). Let \(\tilde{r} > 1\). By Gromov's inequality, if \(\|v\|_{\sup, k\phi} \geq k^{2n}\), then for sufficiently large \(k\), we have the uniform bound
\begin{equation}\label{estimate11}
\|v\|_{L^2, k\phi} \geq \tilde{r} \sqrt{\frac{n_k+2}{2\pi}}.
\end{equation}
The exponential decay of \(f(k)\) then follows from \cite[Proposition 3.2.2]{BostTheta}:
\begin{equation}\label{f(k)}
f(k) \leq \left[\tilde{r} \, e^{-\frac{1}{2}(\tilde{r}^2-1)}\right]^k.
\end{equation}

To compare the measures, we analyze the terms in two parts:

\begin{enumerate}[label=(\roman*)] 
\item \emph{Contribution of large sections:} 
Using Gromov's inequality and the estimate \eqref{estimate11}, for sufficiently large \( k \), we have:
\[
\begin{aligned}
\sum_{\substack{v\in H^0(\mathcal{X}, k\mathcal{L}) \\ \|v\|_{\sup, k\phi} \geq k^{2n}}} \|v(x)\|_{k\phi}^2 \frac{e^{-\pi \|v\|_{L^2, k\phi}^2}}{ \sum_{u} e^{-\pi \|u\|^2}} 
&\leq C k^n \sum_{\substack{v \\ \|v\|_{L^2} \geq \tilde{r} \sqrt{\frac{n_k + 2}{2\pi}}}} \|v\|_{L^2, k\phi}^2 \frac{e^{-\pi \|v\|_{L^2, k\phi}^2}}{ \sum_{u} e^{-\pi \|u\|^2}} \\
&\leq C k^n \left( \tilde{r}^2 e^{-(\tilde{r}^2-1)} \right)^{\frac{2 + n_k}{2}},
\end{aligned}
\]
where the last inequality follows from \eqref{Bineq2}. This term decays exponentially as \(k \to \infty\).

\item \emph{Contribution of small sections:} 
Applying Lemma \ref{lemma2} and the definition of \(f(k)\):
\[
\begin{aligned}
\sum_{\substack{v\in H^0(\mathcal{X}, k\mathcal{L}) \\ \|v\|_{\sup, k\phi} < k^{2n}}} &\|v(x)\|_{k\phi}^2 \frac{e^{-\pi \|v\|_{L^2, k\phi}^2}}{\sum_{\|u\|_{\sup} < k^{2n}} e^{-\pi \|u\|^2}} \\
&= \frac{1}{1 - f(k)} \sum_{\substack{v \\ \|v\|_{\sup} < k^{2n}}} \|v(x)\|_{k\phi}^2 \frac{e^{-\pi \|v\|_{L^2, k\phi}^2}}{\sum_{u} e^{-\pi \|u\|_{L^2, k\phi}^2}} \\
&\leq \frac{C k^n}{2\pi (1 - f(k))} n_k,
\end{aligned}
\]
where \(C\) is independent of \(k\).
\end{enumerate}

By decomposing the full sum and applying the estimates above, we obtain:
\begin{equation}\label{finitesupport}
\frac{1}{k^n} \Theta(\mu, k\phi) = \frac{1}{k^n} \sum_{\substack{v\in H^0(\mathcal{X}, k\mathcal{L}) \\ \|v\|_{\sup, k\phi} < k^{2n}}} \|v(x)\|_{k\phi}^2 \frac{e^{-\pi \|v\|_{L^2, k\phi}^2}}{\sum_{\|u\|_{\sup} < k^{2n}} e^{-\pi \|u\|^2}} + o\left(\tfrac{1}{k^n}\right).
\end{equation}
Replacing \(\overline{H^0(\mathcal{X}, k\mathcal{L})}_{L^2, k\phi}\) with \((\overline{E}_\phi)_{k}\) and repeating this truncation argument yields:
\begin{equation}\label{finitesupport2}
\frac{1}{k^n} \rho(\mu, \overline{E}_{\phi, k}) = \frac{1}{k^n} \sum_{\substack{v\in E_{\phi,k} \\ \|v\|_{\sup, k\phi} < k^{2n}}} \|v(x)\|_{k\phi}^2 \frac{e^{-\pi \|v\|^2}}{\sum_{u \in E_{k,t}, \|u\|_{\sup} < k^{2n}} e^{-\pi \|u\|^2}} + o\left(\tfrac{1}{k^n}\right).
\end{equation}
Combining \eqref{finitesupport} and \eqref{finitesupport2}, we conclude:
\[
\left\| \frac{1}{k^n} \Theta(\mu, k\phi) \mu - \frac{1}{k^n} \rho(\mu, (\overline{E}_\phi)_{k}) 
\right\|_{\sup}\to 0  \ (\text{as } k\to \infty).\]
\end{proof}


\begin{proof}[Proof of Theorem \ref{thm:enveloppe}]
We prove the inequalities for the $\limsup$ measures; the $\liminf$ case follows identically.

\medskip

Fix $t>0$. By Theorem~\ref{Thetafinitesum}, we have the pointwise asymptotic identity
\begin{equation}\label{eq:13_1}
\frac{1}{k^n}\Theta(\mu,k\phi)
=
\frac{1}{k^n}\rho(\mu,\overline{E}_{\phi,k}{})
+ o(1),
\qquad k\to\infty,
\end{equation}
where $E_{\phi,k}{}\subset H^0(\mathcal X,k\mathcal L)$ denotes the truncated submodule defined by
\[
E_{\phi,k}{}
=
\left\langle
\{\, s \in H^0(\mathcal X,k\mathcal L)
\mid \|s\|_{\sup,k\phi} \le k^{2n} \}
\right\rangle .
\]
The error term $o(1)$ is uniform in $x\in\mathcal X(\mathbb C)$.

\medskip

For any section $s\in H^0(\mathcal X,k\mathcal L)$, we have
\[
\|s\|_{\sup,k{(\phi+t)}}
=
e^{-kt}\|s\|_{\sup,k\phi}.
\]
Hence,
\[
\|s\|_{\sup,k\phi} \le k^{2n}
\quad \Longrightarrow \quad
\|s\|_{\sup,k{(\phi+t)}} \le e^{-kt}k^{2n}.
\]
Since $e^{-kt}k^{2n} < 1$ for $k\gg1$, we obtain
\[
E_{\phi,k}{}
\subset
S_{\phi+t,k}{}
\qquad (k\gg1).
\]
It is clear that
\[
S_{\phi,k}\subset E_{\phi,k}\quad (\forall k).
\]

\medskip

If $F_k \subset G_k$ are normed submodules, then
\[
\rho(\mu,\overline{F}_k)(x)
\le
\rho(\mu,\overline{G}_k)(x)
\qquad \forall x\in\mathcal X(\mathbb C).
\]
Applying this to the inclusions above yields, for $k\gg1$,
\[
\rho(\mu,\overline{S}_{\phi,k}{})\leq
\rho(\mu,\overline{E}_{\phi,k}{})
\le
\rho(\mu,\overline{S}_{(\phi+t),k}{}).
\]

\medskip

Taking the pointwise $\limsup$ and multiplying by $\mu$, we conclude
\[
\widehat{\omega}^{\sup}_{\mathrm{eq}}\big((\mathcal{X},\mu), \overline{S}_{\phi,\bullet}\big)
\le
\widehat{\mu}^{\sup}_{\mathrm{eq}}\left((\mathcal X,\mu),\overline{\mathcal L}_\phi\right)
\le
\widehat{\omega}^{\sup}_{\mathrm{eq}}\big((\mathcal{X},\mu), \overline{S}_{\phi+t,\bullet}\big).
\]

\medskip

Since all inequalities above are pointwise, the same argument applied to $\liminf$ gives
\[
\widehat{\omega}^{\inf}_{\mathrm{eq}}\left((\mathcal X,\mu),\overline{S}_{\phi+t,\bullet}\right)
\le
\widehat{\mu}^{\inf}_{\mathrm{eq}}\left((\mathcal X,\mu),\overline{\mathcal L}_\phi\right)
\le
\widehat{\omega}^{\inf}_{\mathrm{eq}}\left((\mathcal X,\mu),\overline{S}_{\phi+t,\bullet}\right).
\]

This completes the proof.
\end{proof}

\begin{remark}
We shall prove a refinement of this theorem in Theorem \ref{thm:13_2} in toric setting. Namely, we shall show

\[
\widehat{\omega}^{\inf}_{\mathrm{eq}}\left((\mathcal X,\mu),\overline{S}_{\phi,\bullet}\right)
=
\widehat{\mu}^{\inf}_{\mathrm{eq}}\left((\mathcal X,\mu),\overline{\mathcal L}_\phi\right)
\]
in the toric setting 

\end{remark}


We present the main result, which establishes that the asymptotic  $\Theta$-measure converges to the classical equilibrium measure  $\mu_\phi$. This provides a strong affirmative answer to Conjecture \ref{conj1} for the case of weakly nef.

\begin{theorem}\label{main}
Let $\mathcal{X}$ be a smooth arithmetic variety over $\mathrm{Spec}(\mathbb{Z})$ of relative dimension $n$. 
Let $\mu$ be a  smooth volume form on $\mathcal{X}(\mathbb{C})$. 
Let $\overline{\mathcal{L}}_\phi$ be a smooth weakly nef Hermitian line bundle on $\mathcal{X}$ and assume that $\mathcal{L}$ is ample. 
Then
\[
\widehat{\mu}^{\sup}_{\mathrm{eq}}\!\left((\mathcal{X},\mu),\overline{\mathcal{L}}_\phi\right)
=
\widehat{\mu}^{\inf}_{\mathrm{eq}}\!\left((\mathcal{X},\mu),\overline{\mathcal{L}}_\phi\right)
=
\mu_\phi .
\]
In particular, Conjecture~\ref{conj1} holds in this case.
\end{theorem}

\begin{proof}

Let $\psi$ be a smooth weight on $\mathcal{L}$ such that $\overline{\mathcal{L}}_\psi$ is ample. 
Then $k\overline{\mathcal{L}}_\phi+\overline{\mathcal{L}}_\psi$ is ample for every $k\ge 1$. 
In particular,
\[
\lambda_1\!\left(
\overline{H^0(\mathcal{X},(k+1)\mathcal{L})}^{\,\vee}_{L^2,k\phi+\psi}
\right)
\ge 1 .
\]

Since $\mathcal{L}$ is ample, the sequence
\[
\left(
\frac{1}{k}
\log 
\lambda_1\!\left(
\overline{H^0(\mathcal{X},(k+1)\mathcal{L})}^{\,\vee}_{\sup,k\phi+\psi}
\right)
\right)_{k\ge 1}
\]
converges to a finite limit $> 0$. 

Moreover, since $\psi-\phi$ is bounded on $\mathcal{X}(\mathbb{C})$, 
there exists $c>0$ such that
\[
k\phi+\phi-c \le k\phi+\psi \le k\phi+\phi+c .
\]
This implies
\[
\|\cdot\|_{L^2,k\phi+\phi+c}
\le
\|\cdot\|_{L^2,k\phi+\psi}
\le
\|\cdot\|_{L^2,k\phi+\phi-c}.
\]
Consequently,
\[
\left|
\frac{1}{k}
\log 
\lambda_1\!\left(
\overline{H^0(\mathcal{X},k\mathcal{L})}^{\,\vee}_{L^2,(k+1)\phi}
\right)
-
\frac{1}{k}
\log 
\lambda_1\!\left(
\overline{H^0(\mathcal{X},k\mathcal{L})}^{\,\vee}_{L^2,k\phi+\psi}
\right)
\right|
\le \frac{c}{k}.
\]
Hence the sequence
\[
\left(
\frac{1}{k}
\log 
\lambda_1\!\left(
\overline{H^0(\mathcal{X},k\mathcal{L})}^{\,\vee}_{L^2,k\phi}
\right)
\right)_{k\ge 1}
\]
converges to a finite limit $> 0$.  
In particular, for $k\gg 1$,
\[
\lambda_1\!\left(
\overline{H^0(\mathcal{X},k\mathcal{L})}^{\,\vee}_{L^2,k\phi}
\right)
\ge e^{\varepsilon k}
\]
for some $\varepsilon>0$.

\medskip

Fix $x\in\mathcal{X}(\mathbb{C})$. 
For $s\ge 0$, endow $H^0(\mathcal{X},k\mathcal{L})$ with the norm
\[
\|v\|_{k,s}^2
=
\|v\|_{L^2,k\phi}^2
+
s\,|v(x)|_\phi^2 .
\]
Denote the resulting Euclidean lattice by $\overline{V}_{k,s}$. 
By the Poisson summation formula,
\begin{equation}\label{poisson1}
\sum_{v} e^{-\pi\|v\|_{k,s}^2}
=
\frac{1}{\mathrm{covol}(\overline{V}_{k,s})}
\sum_{v^\vee}
e^{-\pi\|v^\vee\|_{k,s}^{\vee 2}} .
\end{equation}

Choose a $\mathbb{Z}$-basis of $H^0(\mathcal{X},k\mathcal{L})$ and let
\[
M_k=(\langle s_i,s_j\rangle_{L^2,k\phi}),
\qquad
P_k(x)=(\langle s_i(x),s_j(x)\rangle_{k\phi}).
\]
Then $\|v\|_{k,s}^2$ corresponds to the quadratic form $M_k+sP_k(x)$, and \eqref{poisson1} becomes
\[
\sum_{v} e^{-\pi(\|v\|_{k\phi}^2+s|v(x)|_\phi^2)}
=
\frac{1}{\det(M_k+sP_k(x))^{1/2}}
\sum_{a\in\mathbb{Z}^{N_k}}
e^{-\pi a^t(M_k+sP_k(x))^{-1}a}.
\]

Differentiating with respect to $s$ at $s=0$ yields
\begin{equation}\label{theta-decomp}
\Theta(\mu,k\phi)(x)
=
\rho(\mu,k\phi)(x)
-
R_k(x),
\end{equation}
where
\[
R_k(x)
=
\frac{\sum_{a}
(a^t M_k^{-1}P_k(x)M_k^{-1}a)\,
e^{-\pi a^tM_k^{-1}a}}
{\sum_{a} e^{-\pi a^tM_k^{-1}a}} .
\]

\medskip

Let $f$ be continuous on $\mathcal{X}(\mathbb{C})$. 
Integrating \eqref{theta-decomp} gives
\[
\int f\,\Theta(\mu,k\phi)\,\mu
=
\int f\,\rho(\mu,k\phi)\,\mu
-
\int f\,R_k\,\mu .
\]
Our next step is to show that the remainder  is a $o(k^n)$  as $k\to\infty$.

A direct computation shows
\[
\int_{\mathcal{X}(\mathbb{C})}
(a^t M_k^{-1}P_k(x)M_k^{-1}a)
\,\mu(x)
=
a^t M_k^{-1}a .
\]
Hence
\[
\left|
\int f\,R_k\,\mu
\right|
\le
\|f\|_\infty
\frac{\sum_{a}
(a^tM_k^{-1}a)\,
e^{-\pi a^tM_k^{-1}a}}
{\sum_{a} e^{-\pi a^tM_k^{-1}a}}\]
\[= \|f\|_{\infty}\sum_{\sigma \in H^0(\mathcal{X}, k\mathcal{L})^\vee} \|\sigma\|^{\vee 2}_{L^2, k\phi} \frac{e^{-\pi \|\sigma\|^{\vee 2}_{L^2, k\phi}}}{\sum_{\eta} e^{-\pi \|\eta\|^{\vee 2}_{L^2, k\phi}}}.
\]

Using the exponential lower bound for $\lambda_1$ and inequality~\eqref{Bineq2}, 
the right-hand side is $O(e^{-\varepsilon k})$. 
In particular,
\[
\lim_{k\to\infty}
\frac{1}{k^n}
\int f\,R_k\,\mu
=
0 .
\]

\medskip

By the asymptotic expansion of the Bergman kernel (see \eqref{B1}),
\[
\frac{1}{k^n}\rho(\mu,k\phi)\,\mu
\longrightarrow
\mu_\phi
\]
weakly as $k\to\infty$. 
Since the remainder term is negligible after normalization by $k^n$, 
we conclude
\[
\frac{1}{k^n}\Theta(\mu,k\phi)\,\mu
\longrightarrow
\mu_\phi .
\]

This proves the equality of $\widehat{\mu}^{\sup}_{\mathrm{eq}}$ and $\widehat{\mu}^{\inf}_{\mathrm{eq}}$ 
and establishes the theorem.
\end{proof}


\section{Arithmetic theta invariants on arithmetic toric varieties}

In this section, we study the arithmetic distortion function in the toric setting. 
Several results obtained in the general framework admit simpler proofs or sharper refinements when restricted to toric varieties.

In particular, we prove Theorem~\ref{thm:13_2}, which refines Theorem~\ref{thm:enveloppe} in the toric case. 
More precisely, one may take $t=0$ in the envelope inequality, yielding
\[
\widehat{\mu}^{\sup}_{\mathrm{eq}}\big((\mathcal{X},\mu), \overline{\mathcal{L}}_\phi\big)
=
\widehat{\omega}^{\sup}_{\mathrm{eq}}\big((\mathcal{X},\mu), \overline{S}_{\phi,\bullet}\big),
\]
and
\[
\widehat{\mu}^{\inf}_{\mathrm{eq}}\big((\mathcal{X},\mu), \overline{\mathcal{L}}_\phi\big)
=
\widehat{\omega}^{\inf}_{\mathrm{eq}}\big((\mathcal{X},\mu), \overline{S}_{\phi,\bullet}\big).
\]

Furthermore, Theorems~\ref{toricThetaNef} and~\ref{thm:theta-toric-1} constitute the main results of this section and provide substantial evidence in support of Conjecture~\ref{conj1}.

\medskip

We begin by recalling the necessary background from the theory of toric varieties.

\subsection*{Basic preliminaries on toric varieties}

Let $Q$ be a free $\Z$-module of rank $n$ and 
$P := \mathrm{Hom}_\Z(Q,\Z)$ its dual. 
Let $\Sigma$ be a fan in $Q_\R := Q \otimes_\Z \R$ and 
$\X := X_\Sigma$ the associated toric variety over $\mathrm{Spec}(\Z)$ 
(see \cite{Oda}). 
Throughout this section we assume that $\X$ is smooth and projective, 
i.e.\ $\Sigma$ is smooth and $|\Sigma|=Q_\R$.

The torus is
\[
\T_Q := \mathrm{Hom}_\Z(P,\C^\ast) \simeq (\C^\ast)^n,
\]
with maximal compact subgroup $\mathbb{S}_Q \simeq (\mathbb{S}^1)^n$. 
There is a dense open immersion $\T_Q \hookrightarrow X$ 
compatible with the torus action.

\subsubsection*{Divisors and polytopes}

Let $D$ be a toric Cartier divisor. 
It is determined by a $\Sigma$-linear support function
\[
\Psi_D : Q_\R \to \R,
\]
and defines the polytope
\[
\Delta_D
=
\{ x \in P_\R \mid 
\langle x,u\rangle \ge \Psi_D(u),\ \forall u\in Q_\R \}.
\]

One has
\[
H^0(\X,\mathcal{O}(D))
=
\bigoplus_{m\in\Delta_D\cap P}
\Z\,\chi^m ,
\]
and
\[
\mathrm{vol}(\mathcal{O}(D)) = n!\,\mathrm{vol}(\Delta_D).
\]

Moreover:
\begin{itemize}
\item $D$ is nef $\iff$ $\Psi_D$ is concave,
\item $D$ is ample $\iff$ $\Psi_D$ is strictly concave,
\item $D$ is big $\iff$ $\dim(\Delta_D)=n$.
\end{itemize}

\subsubsection*{Toric metrics}

The support function $\Psi_D$ defines the canonical continuous metric 
$\|\cdot\|_{\infty,D}$ on $\mathcal{O}(D)$. 
If $\mathcal{O}(D)$ is globally generated, this metric is semipositive.

Let $\|\cdot\|_{\overline{\LL}_\phi}$ be an $\mathbb{S}_Q$-invariant Hermitian metric 
on $\mathcal{O}(D)$ such that 
\(
\|\cdot\|_{\overline{\LL}_\phi}/\|\cdot\|_{\infty,D}
\)
is bounded. 
Write
\[
\overline{\LL}_\phi := (\mathcal{O}(D),\|\cdot\|_{\overline{\LL}_\phi}).
\]
We say that $\overline{\LL}_\phi $ is a \emph{toric Hermitian line bundle}.

\medskip

Via the exponential parametrization
\[
\exp(-(\cdot)) : Q_\R \to \T_Q \subset X(\C),
\]
define
\[
g_\phi(u)
:=
\log \|s_D(\exp(-u))\|_{\overline{\LL}_\phi},
\qquad u\in Q_\R.
\]

\begin{definition}
The Legendre–Fenchel transform of $g_\phi$ is
\[
\check g_\phi(x)
=
\inf_{u\in Q_\R}
\bigl( \langle x,u\rangle - g_\phi(u) \bigr),
\qquad x\in P_\R.
\]
\end{definition}

Then $\check g_\phi$ is finite precisely on $\Delta_D$ and is concave there.

\medskip

Assume now that $\overline{\LL}_\phi$ is smooth and positive. 
Writing $z=\exp(-u-i\theta)$ with $u,\theta\in Q_\R$, 
one computes
\[
\frac{\partial^2}{\partial z_k \partial \overline{z}_l}
\bigl(-\log\|s_D\|^2\bigr)
=
\frac{1}{z_k\overline{z}_l}
\frac{\partial^2 g_\phi}{\partial u_k \partial u_l}.
\]
Hence $g_\phi$ is smooth and strictly concave on $Q_\R$.

Consequently,
\[
x=\nabla g_\phi(u)
\]
defines a $\mathscr C^\infty$-diffeomorphism
\[
Q_\R \xrightarrow{\sim} \mathrm{Int}(\Delta_D),
\]
and for $x\in\mathrm{Int}(\Delta_D)$
\[
\check g_\phi(x)
=
\langle x,G_\phi(x)\rangle
-
g_\phi(G_\phi(x)),
\]
where $G_\phi=(\nabla g_\phi)^{-1}$
(see \cite[Thm.\ 26.5]{convex}).

\subsection*{Arithmetic theta invariants on toric varieties}

Since the global sections of $\mathcal O(D)$ are given by characters $\chi^m$ which are mutually orthogonal with respect to any $\mathbb S_Q$-invariant $L^2$-norm, the lattice sum appearing in the Poisson formula diagonalizes. 
As a consequence, $\Theta(\mu,\phi)$ decomposes into a sum of one–dimensional Riemann theta $\theta(t)$ contributions indexed by lattice points of the polytope.

\begin{proposition}\label{1331}
Let $\mu$ be a normalized smooth $\mathbb S_Q$-invariant volume form on $\X(\C)$
and let $\phi$ be a toric weight on $\LL=\mathcal O(D)$.
Let $E\subset H^0(\X,\LL)$ be a sublattice of rank $r_E$
generated by characters $\{\chi^{m_1},\dots,\chi^{m_{r_E}}\}$.
Then for every $x\in \X(\C)$:

\begin{enumerate}
\item One has
\[
\Theta(\mu,\overline E)(x)
=
-2 \sum_{i=1}^{r_E}
\|\chi^{m_i}(x)\|_\phi^2
\frac{\theta'\!\left(\|\chi^{m_i}\|_{(L^2,\phi)}^2\right)}
{\theta\!\left(\|\chi^{m_i}\|_{(L^2,\phi)}^2\right)}.
\]
In particular, for $E=H^0(\X,\LL)$,
\[
\Theta(\mu,\phi)(x)
=
-2 \sum_{m\in\Delta_D\cap P}
\|\chi^m(x)\|_\phi^2
\frac{\theta'\!\left(\|\chi^m\|_{(L^2,\phi)}^2\right)}
{\theta\!\left(\|\chi^m\|_{(L^2,\phi)}^2\right)}.
\]

\item Moreover,
\[
\Theta(\mu,\overline E)(x)
\le
\rho(\mu,\overline E)(x).
\]
In particular,
\[
\Theta(\mu,\phi)(x)
\le
\rho(\mu,\phi)(x).
\]
\end{enumerate}
\end{proposition}

\begin{proof}
Since $\mu$ is $\mathbb S_Q$-invariant, the characters
$\{\chi^{m_i}\}_{i=1}^{r_E}$ form an orthogonal family for the $L^2$-norm.
Every section $v\in E$ can therefore be written uniquely as
\[
v=\sum_{i=1}^{r_E} a_i \chi^{m_i},
\qquad a_i\in\Z,
\]
and
\[
\|v\|_{L^2,\phi}^2
=
\sum_{i=1}^{r_E} a_i^2
\|\chi^{m_i}\|_{L^2,\phi}^2.
\]

\medskip

\noindent
(1)
By definition,
\[
\sum_{v\in E}
\|v(x)\|_\phi^2
e^{-\pi\|v\|_{L^2,\phi}^2}
=
\sum_{a\in\Z^{r_E}}
\Bigl|
\sum_i a_i\chi^{m_i}(x)
\Bigr|^2
e^{-\pi\sum_i a_i^2\|\chi^{m_i}\|^2}.
\]
Expanding the square and using the symmetry of the Gaussian,
the mixed terms vanish, and we obtain
\[
=
\sum_{i=1}^{r_E}
\|\chi^{m_i}(x)\|_\phi^2
\sum_{a\in\Z^{r_E}}
a_i^2
e^{-\pi\sum_j a_j^2\|\chi^{m_j}\|^2}.
\]

Since the quadratic form is diagonal, the lattice sum factorizes:
\[
\sum_{a\in\Z^{r_E}}
a_i^2
e^{-\pi\sum_j a_j^2\|\chi^{m_j}\|^2}
=
\left(
\sum_{a_i\in\Z}
a_i^2 e^{-\pi a_i^2\|\chi^{m_i}\|^2}
\right)
\prod_{j\neq i}
\theta(\|\chi^{m_j}\|^2),
\]
where $\theta(t)=\sum_{n\in\Z}e^{-\pi n^2 t}$.

Using
\[
\sum_{n\in\Z} n^2 e^{-\pi n^2 t}
=
-\frac{1}{\pi}\theta'(t),
\]
we deduce
\[
\sum_{v\in E}
\|v(x)\|_\phi^2
e^{-\pi\|v\|_{L^2,\phi}^2}
=
-\frac{1}{\pi}
\left(
\sum_{i=1}^{r_E}
\|\chi^{m_i}(x)\|_\phi^2
\frac{\theta'(\|\chi^{m_i}\|^2)}
{\theta(\|\chi^{m_i}\|^2)}
\right)
\prod_{j=1}^{r_E}
\theta(\|\chi^{m_j}\|^2).
\]

Dividing by $\prod_j \theta(\|\chi^{m_j}\|^2)$
and multiplying by $2\pi$ (definition of $\Theta$)
gives the stated formula.

\medskip

\noindent
(2)
The classical inequality
\[
-\theta'(t)\le \frac{\theta(t)}{2t},
\qquad t>0,
\]
implies
\[
-2\frac{\theta'(t)}{\theta(t)}\le \frac{1}{t}.
\]
Applying this termwise in (1) yields
\[
\Theta(\mu,\overline E)(x)
\le
\sum_{i=1}^{r_E}
\frac{\|\chi^{m_i}(x)\|_\phi^2}
{\|\chi^{m_i}\|_{L^2,\phi}^2}.
\]
The right-hand side is precisely the Bergman kernel
$\rho(\mu,\overline E)(x)$ in the orthogonal basis
$\{\chi^{m_i}\}$.
\end{proof}

\begin{lemma}\label{lem:theta-estimate}
There exists a constant $C>0$ such that for all $t>0$ one has
\[
\left|\, t\frac{\theta'(t)}{\theta(t)} \right|
\le C \max(1,t) e^{-t}.
\]
In particular,
\[
-\, t\frac{\theta'(t)}{\theta(t)}
\le C \max(1,t) e^{-t}.
\]
\end{lemma}

\begin{proof}
We treat separately the cases $t\ge 1$ and $0<t\le 1$.

\medskip
\noindent
\textsc{Case 1: $t\ge 1$.}
By the exponential decay of $\theta'(t)$ at infinity, there exists a constant
$A>0$ such that
\[
|\theta'(t)| \le A e^{-t}
\qquad \text{for all } t\ge 1.
\]
Since $\theta(t)$ is positive and continuous, we have
$\theta(t)\ge \theta(1)>0$ for $t\ge 1$. Hence
\[
\left| t\frac{\theta'(t)}{\theta(t)} \right|
\le \frac{A}{\theta(1)}\, t e^{-t}
\qquad (t\ge 1).
\]

\medskip
\noindent
\textsc{Case 2: $0<t\le 1$.}
Using the functional equation
\[
\frac{\theta'(t)}{\theta(t)}
= -\frac{1}{2t}
- \frac{\theta'(1/t)}{t^2 \theta(1/t)},
\]
we obtain
\[
-\, t\frac{\theta'(t)}{\theta(t)}
= \frac{1}{2}
+ \frac{\theta'(1/t)}{t\,\theta(1/t)}.
\]
Since $1/t\ge 1$, the estimate from Case~1 applied to $1/t$ gives
\[
\left|
\frac{\theta'(1/t)}{t\,\theta(1/t)}
\right|
\le C_1 e^{-1/t}
\]
for some constant $C_1>0$. Therefore
\[
\left| t\frac{\theta'(t)}{\theta(t)} \right|
\le \frac12 + C_1 e^{-1/t}
\qquad (0<t\le 1).
\]
As $e^{-1/t}\le e^{-t}$ on $(0,1]$, we deduce
\[
\left| t\frac{\theta'(t)}{\theta(t)} \right|
\le C_2
\le C_2 e^{-t}
\qquad (0<t\le 1),
\]
for a suitable constant $C_2>0$.

\medskip
Combining the two cases, and enlarging the constant if necessary,
we obtain
\[
\left| t\frac{\theta'(t)}{\theta(t)} \right|
\le C \max(1,t) e^{-t}
\qquad \text{for all } t>0,
\]
which concludes the proof.
\end{proof}

In the next theorem, we make  the relation between $\Theta(\mu,\phi)$ and the convex geometry of $\Delta_D$  transparent. 
In particular, only those characters whose sup-norms remain effectively bounded contribute asymptotically; 
equivalently, those $m$ for which the Legendre transform $\check g_\phi(m/k)$ is non-negative. 
The remaining terms are exponentially suppressed as $k\to\infty$.

\begin{theorem}\label{thm:13_2}
Let $\phi$ be a smooth toric weight on $\mathcal L=\mathcal O(D)$ and set
\[
\overline{S}_{\phi,k}
:=\overline{
\left\langle 
\left\{ 
s \in H^0(\mathcal{X},k\mathcal{L})
\;\middle|\;
\|s\|_{\sup,k\phi} < 1
\right\}
\right\rangle},
\]
the normed $\mathbb Z$-module generated by sections of $\sup$-norm $<1$,
endowed with the induced norm of 
$\overline{H^0(\X,k\LL)}_{(L^2,k\phi)}$. Then, as measures on $\mathcal X(\C)$,
\[
\widehat{\mu}^{\sup}_{\mathrm{eq}}\big((\mathcal{X},\mu), \overline{\mathcal{L}}_\phi\big)
=\widehat{\mu}^{\sup}_{\mathrm{eq}}\big((\mathcal{X},\mu), \overline{S}_{\phi,\bullet}\big)=
\widehat{\omega}^{\sup}_{\mathrm{eq}}\big((\mathcal{X},\mu), \overline{S}_{\phi,\bullet}\big),
\]

\end{theorem}

\begin{proof}

Let us first observe that the $\mathbb{Z}$-module 
$\overline{S}_{\phi,k}$ is generated by the characters $\chi^m$ satisfying
\[
\|\chi^m\|_{(\sup,k\phi)} < 1,
\qquad
m \in (k\Delta_D)\cap \mathbb{Z}^n.
\]
These characters are mutually orthogonal with respect to the $L^2$-norm, hence they are linearly independent and form a free $\mathbb{Z}$-basis of $S_{\phi,k}$.

By Proposition~\ref{1331}, it follows clearly that for every $x\in\mathcal{X}(\mathbb{C})$,
\[
\Theta(\mu,\overline{S}_{\phi,k})(x)
\le
\Theta(\mu,k\phi)(x).
\]

To prove the desired asymptotic equality, we split the sum defining 
$\Theta(\mu,k\phi)(x)$ according to the sign of 
$\check g_{\overline D}(m/k)$. 
It therefore suffices to show that
\[
\limsup_{k\to\infty}
\int_{\mathcal X(\mathbb{C})}
f \frac{\Theta(\mu,k\phi)}{k^n}\,\mu
=
\limsup_{k\to\infty}
\int_{\mathcal X(\mathbb{C})}
f \frac{\rho(\mu,\overline S_{\phi,k})}{k^n}\,\mu
\]
for every $f\in \mathscr{C}^0(\mathcal X(\mathbb{C}))$.

\medskip

\medskip
\noindent
\textsc{Step 1: Contribution of $\check g_{\overline D}>0$ in the expansion of $\Theta(\mu,k\phi)$.}

Using the modular identity
\[
\frac{\theta'(t)}{\theta(t)}
=
-\frac{1}{2t}
-
\frac{\theta'(1/t)}{t^2\theta(1/t)},
\]
we obtain
\[
\Theta(\mu,k\phi)
=
\sum_{\check g\ge0}
\frac{\|\chi^m(x)\|_{k\phi}^2}
{\|\chi^m\|_{L^2,k\phi}^2}
+
R_k(x),
\]
where $R_k$ is the remainder term involving
$\theta'(1/t)$.

By Lemma~\ref{lem:theta-estimate},
\[
\left|
t\frac{\theta'(t)}{\theta(t)}
\right|
\le
C\max(1,t)e^{-t}.
\]

If $\check g_{\overline D}(m/k)\ge\varepsilon>0$,
then
\[
\|\chi^m\|_{\sup,k\phi}^{-2}
=
\exp(2k\check g_{\overline D}(m/k))
\ge
e^{2k\varepsilon}.
\]
By Gromov’s inequality,
\[
\|\chi^m\|_{L^2,k\phi}
\ge
C k^{-n} e^{k\varepsilon}.
\]
Hence
\[
\max(1,t)e^{-t}
=
O\!\left(
e^{-c e^{2k\varepsilon}}
\right),
\]
so that the total contribution of these terms is exponentially small,
uniformly in $x$.

If $0<\check g_{\overline D}(m/k)\le\varepsilon$,
their number is
\[
O\!\left(
\varepsilon k^n
\right),
\]
since it corresponds to lattice points in an $\varepsilon$-neighbourhood
of the hypersurface $\{\check g_{\overline D}=0\}$.
Thus their total contribution is $O(\varepsilon k^n)$.

Letting first $k\to\infty$ and then $\varepsilon\to0$,
only the principal term
\[
\sum_{\check g>0}
\frac{\|\chi^m(x)\|_{k\phi}^2}
{\|\chi^m\|_{L^2,k\phi}^2}
=
\rho(\mu,\overline S_{\phi,k})(x)
\]
remains.

\medskip
\noindent
\textsc{Step 2: Contribution of $\check g_{\overline D}\leq 0$.}

If $\check g_{\overline D}(m/k)<-\varepsilon$,
then
\[
\|\chi^m\|_{L^2,k\phi}
\ge
C k^{-n} e^{k\varepsilon},
\]
and Lemma~\ref{lem:theta-estimate} yields again
super-exponential decay in $k$.
Hence this part vanishes after division by $k^n$.

If $-\varepsilon\le\check g_{\overline D}(m/k)\leq 0$,
their number is $O(\varepsilon k^n)$,
so their total contribution to
\[
\int_X f \frac{\Theta(\mu,k\phi)}{k^n} {{\mu}}
\]
is bounded by
$\|f\|_\infty O(\varepsilon)$.

\medskip

Combining the above estimates,
we obtain
\[
\limsup_{k\to\infty}
\int_X f
\frac{\Theta(\mu,k\phi)}{k^n}{{\mu}}
=
\limsup_{k\to\infty}
\int_X f
\frac{\rho(\mu,\overline S_{\phi,k})}{k^n}{{\mu}} .
\]

Since this holds for all $f\in \mathscr{C}^0(\mathcal X(\C))$,
the weak convergence of measures follows.
\end{proof}

The following two results relate the $\Theta$-measure to the arithmetic volume 
and to the equilibrium measure in the toric setting. They complete the passage 
from discrete lattice sums to continuous measures on $\mathcal X(\C)$.

\begin{corollary}\label{limmeasure}
With the same notations as above, one has
\[
\limsup_{k\to\infty}
\int_X \frac{\Theta(\mu,k\phi)}{k^n}\, {{\mu}}
=
d\mathrm{vol}(\overline{\mathcal L}_\phi).
\]
\end{corollary}

\begin{proof}
By Theorem~\ref{thm:13_2}, we have the weak equality of measures
\[
\limsup_{k\to\infty}
\frac{\Theta(\mu,k\phi)}{k^n}\, {{\mu}}
=
\limsup_{k\to\infty}
\frac{\rho(\mu,\overline S_{\phi,k})}{k^n}\, {{\mu}} .
\]
Integrating against the constant function $1$ gives
\[
\limsup_{k\to\infty}
\int_X
\frac{\Theta(\mu,k\phi)}{k^n}\, {{\mu}}
=
\limsup_{k\to\infty}
\frac{1}{k^n}
\int_X \rho(\mu,\overline S_{\phi,k})\, {{\mu}} .
\]
The right-hand side equals the arithmetic volume
$d\mathrm{vol}(\overline{\mathcal L}_\phi)$ by definition of $\overline S_{\phi,k}$ 
and the asymptotic growth of sections.
\end{proof}

\begin{theorem}\label{toricThetaNef}
Let $\mu$ be a smooth invariant volume form and let 
$\overline{\mathcal L}_\phi$ be a smooth toric nef Hermitian line bundle 
on $\mathcal X$. Then, as measures on $\mathcal X(\C)$,
\[
\lim_{k\to\infty}
\frac{1}{k^n}\Theta(\mu,k\phi)\, {{\mu}}
=
\mu_\phi,
\]
where $\mu_\phi$ denotes the equilibrium (Monge–Ampère) measure of $\phi$.
In particular,
\[
d\mathrm{vol}(\overline{\mathcal L}_\phi)
=
\int_{\mathcal X(\C)} \mu_\phi.
\]
\end{theorem}

\begin{proof}
By Theorem~\ref{thm:13_2}, the measures
\[
\frac{1}{k^n}\Theta(\mu,k\phi)\,\mu
\]
have the same weak asymptotics as
\[
\frac{1}{k^n}\rho(\mu,\overline S_{\phi,k})\,\mu .
\]

A careful reading of the proof of Theorem~\ref{thm:13_2} shows that the same conclusion holds when $S_{\phi,k}$ is replaced by the sublattice 
$S_{\phi,k}^{\circ}$ generated by small sections.

For toric nef Hermitian line bundles, we adapt 
\cite[Theorem~6.1 (2)]{Burgos3}, dropping the positivity assumption to fit our setting, in order to show that the Legendre--Fenchel transform associated with $\overline{\mathcal L}_\phi$ is nonnegative on $\Delta$. 
It follows that $S_{\phi,k}^{\circ}$ coincides with the full space of global sections. Hence,
\[
\frac{1}{k^n}\rho(\mu,\overline S_{\phi,k}^\circ)\,\mu
=
\frac{1}{k^n}\rho(\mu,k\phi)\,\mu
\longrightarrow
\mu_\phi
\quad \text{weakly as } k\to\infty .
\]

Therefore the same weak convergence holds for
\[
\frac{1}{k^n}\Theta(\mu,k\phi)\,\mu .
\]
The identity for the  volume derivative then follows from
Corollary~\ref{limmeasure}.
\end{proof}

  The following theorem and proof establish the asymptotic distribution of the arithmetic distortion measure on toric varieties. The result shows that the measure concentrates on a specific subset $U$ of the variety, defined by the non-negativity of the Legendre-Fenchel transform.

 \begin{theorem}\label{thm:theta-toric-1}
Let \( \overline{\mathcal L}_\phi \) be a smooth toric Hermitian line bundle on \( \mathcal X \) such that 
\( c_1(\overline{\mathcal L}_\phi) \) is positive, and set
\[
\mu = c_1(\overline{\mathcal L}_\phi)^n.
\]
Then, as \( k\to\infty \), one has weak convergence of measures
\[
\frac{1}{k^n}\Theta(\mu,k\phi)\,\mu
\longrightarrow
\mathds 1_U \,\mu,
\]
where \( U \) is the closure of
\[
\left\{
z\in \mathcal X^\circ(\C)
\;\middle|\;
\sum_{j=1}^n
z_j\frac{\partial\phi}{\partial z_j}(z)\log|z_j|
-\phi(z)\ge0
\right\}.
\]
\end{theorem}

\begin{proof}

Let $g$ be the Green function attached to $\overline{\mathcal L}_\phi$ and 
$\check g$ its Legendre transform.
For $z\in\mathcal X^\circ$, set
\[
x=\nabla g(-\log|z|).
\]
Then the condition defining $U$ is equivalent to $\check g(x)\ge0$.

We decompose
\[
\mathcal X^\circ(\C)
=
X_1^\circ\sqcup X_2^\circ,
\]
where
\[
X_1^\circ=\{\check g(x)<0\},
\qquad
X_2^\circ=\{\check g(x)\ge0\}.
\]

\medskip
\noindent
\textsc{Step 1: Exponential decay away from the boundary.}

Fix $\varepsilon>0$.

 If $z\in X_1^\circ$, then $\check g(x)<0$.
By strict concavity of $g$, there exists $c_\varepsilon>0$ such that
\[
\check g(y)-\langle y,G(x)\rangle+g(G(x))
\le -c_\varepsilon
\]
for all $y$ with $\check g(y)\ge\varepsilon$.

Using
\[
\|\chi^m\|_{\sup,k\phi}
=
e^{-k\check g(m/k)}
\]
and Gromov's inequality, we obtain
\[
\sum_{\check g(m/k)\ge\varepsilon}
\frac{\|\chi^m(z)\|_{k\phi}^2}
{\|\chi^m\|_{L^2,k\phi}^2}
=
O\!\left(k^{3n}e^{-k c_\varepsilon}\right).
\]
Hence on $X_1^\circ$,
\[
\frac{1}{k^n}\Theta(\mu,k\phi)(z)
\longrightarrow 0.
\]

Similarly, if $z\in X_2^\circ$ and $\check g(m/k)\le -\varepsilon$,
one obtains the same exponential decay.

\medskip
\noindent
\textsc{Step 2: Boundary layer estimate.}

The number of lattice points satisfying
\[
|\check g(m/k)|\le\varepsilon
\]
is
\[
O(\varepsilon k^n),
\]
since this corresponds to lattice points in an
$\varepsilon$–neighbourhood of the hypersurface
$\{\check g=0\}$.

Thus their total contribution to
\[
\frac{1}{k^n}\Theta(\mu,k\phi)
\]
is $O(\varepsilon)$ uniformly in $z$.

\medskip
\noindent
\textsc{Step 3: Main term on $X_2^\circ$.}

On $X_2^\circ$, the dominant contribution comes from
\[
\sum_{\check g(m/k)\ge0}
\frac{\|\chi^m(z)\|_{k\phi}^2}
{\|\chi^m\|_{L^2,k\phi}^2}
=
\rho(\mu,k\phi)(z)
+O(\varepsilon k^n)
+O(k^{3n}e^{-k c_\varepsilon}).
\]

Since
\[
\frac{1}{k^n}\rho(\mu,k\phi)\,\mu
\longrightarrow
\mu
\]
(Bergman kernel asymptotics in the positive toric case),
we deduce that on $X_2^\circ$
\[
\frac{1}{k^n}\Theta(\mu,k\phi)\,\mu
\longrightarrow
\mu.
\]

\medskip

Combining the three steps,
for every continuous function $f$ we obtain
\[
\int_{X}
f\,\frac{1}{k^n}\Theta(\mu,k\phi)\,\mu
=
\int_{X_2^\circ} f\,\mu
+
O(\varepsilon)
+
O(e^{-k c_\varepsilon}).
\]

Letting first $k\to\infty$ and then $\varepsilon\to0$
gives
\[
\int_{X}
f\,\frac{1}{k^n}\Theta(\mu,k\phi)\,\mu
\longrightarrow
\int_{\mathcal X}
f\,\mathds 1_U\,\mu,
\]
which proves the weak convergence.
\end{proof}


\section{Arithmetic volume and its variational properties}\label{sec6}

In this section, we provide a representation of the arithmetic volume of Hermitian line bundles  over  smooth projective arithmetic varieties. This representation utilizes the arithmetic distortion function $\Theta$ and highlights the relationship between the distribution of small sections and the global arithmetic invariants. 

\medskip

Let $\mathcal{X}$ be a smooth projective arithmetic variety over $\mathbb{Z}$, not necessarily toric.

\begin{theorem}\label{Representation}
Let $\overline{\mathcal{L}}_\phi$ be a smooth Hermitian line bundle on $\mathcal{X}$. We have:
\[
\widehat{\mathrm{vol}}(\overline{\mathcal{L}}_\phi)
=
(n+1)!
\limsup_{k\to\infty}
\int_0^{\widehat{\mu}_{\max}(\overline{\mathcal{L}}_\phi)}
\int_{X(\C)}
\frac{1}{k^n}
\Theta(\mu, k(\phi-t))(x)\,\mu\, dt.
\]

\end{theorem}

\begin{proof}
Recall the definition of $h_\theta^0$. For every $u \geq 1$ and $k=1, 2, \ldots$, we consider the variation of the theta invariant under a shift of the metric:
\[
\begin{split}
h_\theta^0(\overline{H^0(\mathcal{X}, k\mathcal{L})}_{(L^2, k\phi)}) - &h_\theta^0(\overline{H^0(\mathcal{X}, k\mathcal{L})}_{(L^2, k\phi - \frac{1}{2} \log u)})\\
& = -\int_1^u \frac{ \frac{d}{dt}\theta_{\overline{H^0(\mathcal{X}, k\mathcal{L})}_{(L^2, k\phi - \frac{1}{2}\log t )}}(1)}{\theta_{\overline{H^0(\mathcal{X}, k\mathcal{L})}_{(L^2, k\phi - \frac{1}{2}\log t )} } (1)} \, dt \\
=& \frac{1}{2} \int_1^u \int_X \Theta(\mu, k\phi)(t; x) \, {{\mu}} \, dt \\
=& \frac{1}{2} \int_1^u \int_X \frac{1}{t} \Theta(\mu, k\phi - \frac{1}{2}\log t) \, {{\mu}} \, dt.
\end{split}
\]
By performing the change of variables $t = s^k$, we obtain:
\[
\frac{1}{2} k \int_1^{u^{\frac{1}{k}}} \int_X \frac{1}{s} \Theta(\mu, k(\phi - \frac{1}{2}\log s)) \, {{\mu}} \, ds.
\]
Further letting $t = \log s$, we have:
\[
\frac{1}{2} k \int_0^{\frac{1}{k} \log u} \int_X \Theta(\mu, k(\phi - \frac{1}{2}t)) \, {{\mu}} \, dt.
\]
Setting $s := \frac{1}{2k} \log u$, we arrive at the identity:
\[
h_\theta^0(\overline{H^0(\mathcal{X}, k\mathcal{L})}_{(L^2, k\phi)}) - h_\theta^0(\overline{H^0(\mathcal{X}, k\mathcal{L})}_{(L^2, k(\phi - s))}) = k \int_0^{s} \int_X \Theta(\mu, k(\phi-t)) \, {{\mu}} \, dt.
\]

Recall that $\widehat{\mu}_{\max}(\overline{\mathcal{L}}_\phi)$ is finite. For any $s > \widehat{\mu}_{\max}(\overline{\mathcal{L}}_\phi)$, it is straightforward that the arithmetic volume satisfies:
\[
\widehat{\mathrm{vol}}(\overline{\mathcal{L}}_{\phi-s}) = 0.
\]
Thus, the term corresponding to the shifted metric does not contribute to the asymptotic growth of the arithmetic volume. We deduce:
\[
\limsup_k \frac{ h_\theta^0(\overline{H^0(\mathcal{X}, k\mathcal{L})}_{(L^2, k\phi)})}{k^{n+1}/(n+1)! } = (n+1)! \limsup_{k\rightarrow \infty} \frac{1}{k^n} \int_0^{s} \int_X \Theta(\mu, k(\phi-t)) \, {{\mu}} \, dt,
\]
for all $s > \widehat{\mu}_{\max}(\overline{\mathcal{L}}_\phi)$. 

For any $\nu > \widehat{\mu}_{\max}(\overline{\mathcal{L}}_\phi)$, the integrand is dominated by the Bergman kernel $\rho(\mu, k\phi)$. Using the Dominated Convergence Theorem and the fact that $\lim_k \frac{1}{k^n}\Theta(\mu, k(\phi-t)) = 0$ for $t > \widehat{\mu}_{\max}$ (see Proposition \ref{vanishingTheta}), we obtain the representation:
\begin{equation}\label{weakrep1}
\limsup_k \frac{ h_\theta^0(\overline{H^0(\mathcal{X}, k\mathcal{L})}_{(L^2, k\phi)})}{k^{n+1}/(n+1)! } = (n+1)! \limsup_{k\rightarrow \infty} \frac{1}{k^n} \int_0^{\widehat{\mu}_{\max}(\overline{\mathcal{L}}_\phi)} \int_X \Theta(\mu, k(\phi-t)) \, {{\mu}} \, dt.
\end{equation}
The proof is concluded by applying Theorem \ref{MorMor}, which relates $h_\theta^0$ to the arithmetic volume.
\end{proof}

The next result shows that arithmetic bigness can be characterized 
by a uniform lower bound on the arithmetic equilibrium measure.

\begin{theorem}\label{thm:big}
Let $\overline{\mathcal{L}}_\phi$ be a continuous Hermitian line bundle on $\mathcal{X}$. 
Then $\overline{\mathcal{L}}_\phi$ is big if and only if there exist constants 
$\varepsilon>0$ and $C>0$ such that, for every $t \in \mathbb{R}$ with $|t|<\varepsilon$, one has
\[
\widehat{\mu}^{\inf}_{\mathrm{eq}}\bigl((\mathcal{X},\mu), \overline{\mathcal{L}}_{\phi+t}\bigr)
\;\ge\;
C\,\mu
\]
as measures on $\mathcal{X}(\mathbb{C})$.
Moreover, the constant $C$ may be chosen independently of $t$.
\end{theorem}

 \begin{proof}
Let us assume that $\overline{\mathcal{L}}_\phi$ is big. Then there exists a non-zero section $s \in H^0(\mathcal{X}, k\mathcal{L})$ with $\|s\|_{\sup, k\phi} < 1$ for some $k \gg 1$. Let $\overline{\mathcal{A}}_\psi$ be an ample Hermitian line bundle on $\mathcal{X}$.

For each $k \in \mathbb{N}$, define the following arithmetic invariants associated with the first minima:
\[
a_k := -\inf_{s \in H^0(\mathcal{X}, k\mathcal{L}-\mathcal{A})} \log \lambda_1\left(\overline{H^0(\mathcal{X}, k\mathcal{L}-\mathcal{A})}_{(\sup, k\phi-\psi)}\right),
\]
\[
b_k := -\inf_{s \in H^0(\mathcal{X}, k\mathcal{L})} \log \lambda_1\left(\overline{H^0(\mathcal{X}, k\mathcal{L})}_{(\sup, k\phi)}\right).
\]
It follows that the sequence $(a_k)$ is superadditive relative to $(b_k)$, i.e., $a_{k+l} \geq a_k + b_l$. By the theory of superadditive sequences (writing $k = ld + r$), we obtain:
\[
\liminf_{k\to \infty} \frac{a_k}{k} \geq \lim_{l\to \infty} \frac{b_l}{l}.
\]
This inequality ensures the existence of an integer $p > 1$ and a section $s \in H^0(\mathcal{X}, p\mathcal{L}-\mathcal{A})$ such that $0 < \|s\|_{\sup, p\phi-\psi} < 1$. By considering the injective map:
\[
H^0(\mathcal{X}, k\mathcal{A}) \longrightarrow H^0(\mathcal{X}, kp\mathcal{L}), \quad \sigma \mapsto \sigma \otimes s^{\otimes k},
\]
we infer that for any $k \in \mathbb{N}_{\geq 1}$ and $|t| \ll 1$:
\[
\rho((\mathcal{X},\mu), k\overline{S}_{(\phi+t),k} ) \geq \rho((\mathcal{X},\mu), kp\overline{\mathcal{A}}_{\psi+t} ).
\]
Taking the limit inferior and applying Theorems \ref{thm:enveloppe} and \ref{main}, we have:
\[
\widehat{\omega}^{\inf}_{\mathrm{eq}}\left((\mathcal{X},\mu), \overline{S}_{\phi+t,\bullet} \right) \geq p^n \widehat{\omega}^{\inf}_{\mathrm{eq}}\left((\mathcal{X},\mu), \overline{\mathcal{A}}_{\psi+t} \right) = p^n \mu_{\psi}.
\]
Thus, $\widehat{\mu}^{\inf}_{\mathrm{eq}}\left((\mathcal{X},\mu), \overline{\mathcal{L}}_{\phi+t} \right) \geq p^n \mu_{\psi}$, which is a positive measure.

\medskip

  From Theorem \ref{Representation}, we easily infer that
\[
\widehat{\mathrm{vol}}(\overline{\mathcal{L}}_\phi)
\geq 
(n+1)!
\liminf_{k\to\infty}
\int_0^{\widehat{\mu}_{\max}(\overline{\mathcal{L}}_\phi)}
\int_{X(\C)}
\hat{\mu}_{\mathrm{eq}}^{\inf}((\X,\mu), \overline{\LL}_{\phi-t})(x)\, dt.
\]
Now, suppose the measure condition holds. The previous inequality shows that
\[
\widehat{\mathrm{vol}}(\overline{\mathcal{L}}_\phi)>0
\]
This demonstrates that $\overline{\mathcal{L}}_\phi$ is big. 
The proof of the theorem is now complete.

\end{proof}


\medskip

Our next result establishes a differentiability property of the arithmetic volume, specifically showing that its directional derivative in the direction of a smooth weight $\eta$ is given by the integral of $\eta$ against the arithmetic equilibrium measure $\mu_{\phi}$.

\begin{theorem}\label{corr271}
Let $\overline{\mathcal{L}}_\phi$ be a weakly nef Hermitian line bundle on $\mathcal{X}$ and assume that $\mathcal{L}$ is ample. Let $\eta$ be a nonnegative smooth weight on $\mathcal{O}$. Then
\[
\lim_{t\to 0^+}
\frac{1}{t}
\left(
\widehat{\mathrm{vol}}\bigl(\overline{\mathcal{L}}_{\phi+t\eta}\bigr)
-
\widehat{\mathrm{vol}}\bigl(\overline{\mathcal{L}}_{\phi}\bigr)
\right)
=(n+1)!
\int_{\mathcal{X}(\mathbb{C})}
\eta \, \mu_{\phi}.
\]
\end{theorem}

\begin{proof}
Let $\phi$ be a smooth weight such that $\overline{\mathcal{L}}_\phi$ is weakly nef. 
For $\varepsilon > 0$, the Hermitian line bundle 
$\overline{\mathcal{L}}_{\phi+\varepsilon}$ is strictly positive. 
Applying Theorem~\ref{main} and Lemma~\ref{difference}, we obtain
\[
\widehat{\mathrm{vol}}\bigl(\overline{\mathcal{L}}_{\phi+\varepsilon+\eta}\bigr)
-
\widehat{\mathrm{vol}}\bigl(\overline{\mathcal{L}}_{\phi+\varepsilon}\bigr)
=(n+1)!
\int_{\mathcal{X}(\mathbb{C})}
\eta
\int_0^1
\mu_{\phi+\varepsilon+s\eta}
\, ds .
\]

By the continuity of the arithmetic volume
(see \eqref{continuitytrivial})
and the weak continuity of the equilibrium measure with respect to the weight
(see \cite[Proposition~4.12]{BermanBoucksom} or \cite[Theorem~2.17]{BEGZ}),
we may let $\varepsilon \to 0$ and obtain
\[
\widehat{\mathrm{vol}}\bigl(\overline{\mathcal{L}}_{\phi+\eta}\bigr)
-
\widehat{\mathrm{vol}}\bigl(\overline{\mathcal{L}}_{\phi}\bigr)
=(n+1)!
\int_{\mathcal{X}(\mathbb{C})}
\eta
\int_0^1
\mu_{\phi+s\eta}
\, ds .
\]

Replacing $\eta$ by $t\eta$ with $t>0$, we deduce
\[
\widehat{\mathrm{vol}}\bigl(\overline{\mathcal{L}}_{\phi+t\eta}\bigr)
-
\widehat{\mathrm{vol}}\bigl(\overline{\mathcal{L}}_{\phi}\bigr)
=(n+1)!
t
\int_{\mathcal{X}(\mathbb{C})}
\eta
\int_0^1
\mu_{\phi+st\eta}
\, ds .
\]

Since $\mu_{\phi+st\eta}$ converges weakly to $\mu_{\phi}$ as $t \to 0$,
and $\eta$ is smooth, we conclude
\[
\lim_{t\to 0^+}
\frac{1}{t}
\left(
\widehat{\mathrm{vol}}\bigl(\overline{\mathcal{L}}_{\phi+t\eta}\bigr)
-
\widehat{\mathrm{vol}}\bigl(\overline{\mathcal{L}}_{\phi}\bigr)
\right)
=(n+1)!
\int_{\mathcal{X}(\mathbb{C})}
\eta \, \mu_{\phi}.
\]
\end{proof}

\section{A volume-based approach to equidistribution}

This section illustrates the machinery developed in this work. 
Traditionally, equidistribution theory is based on higher-dimensional 
height theory \cite{BoGS} and was later generalized in \cite{Zhang}. 
The functoriality of heights plays a crucial role in the proof of 
equidistribution theorems, although it is sometimes left implicit. 
Our treatment is volume-based and avoids the formalism of 
higher-dimensional height theory.

Let $(\mathcal{X}, f, \mathcal{L})$ be an algebraic dynamical system 
defined over $\mathbb{Z}$. Thus $f: \mathcal{X} \to \mathcal{X}$ is a morphism 
and there exists an isomorphism
\[
\alpha: f^*\mathcal{L} \xrightarrow{\sim} \mathcal{L}^{\otimes d},
\]
with $d>1$. We assume that $\mathcal{X}$ is smooth of dimension $n+1$ 
and that $\mathcal{L}$ is ample. 

There exists a unique continuous semipositive metric 
$\|\cdot\|_{\phi_\infty}$ such that
\[
\alpha: f^\ast \overline{\mathcal{L}}_{\phi_\infty} 
\simeq \overline{\mathcal{L}}_{\phi_\infty}^{\otimes d}
\]
is an isometry, see \cite{Zhang}.

Let $(\phi_p)_{p \in \mathbb{N}}$ be a decreasing sequence of smooth 
positive weights on $\mathcal{L}$ such that 
$\overline{\mathcal{L}}_{\phi_p}$ is ample for every $p$ and 
$\phi_p \to \phi_\infty$ uniformly on $\mathcal{X}(\mathbb{C})$. 
Let $\eta$ be a smooth nonnegative function on $\mathcal{X}(\mathbb{C})$. 
Then the Hermitian line bundles 
$\overline{\mathcal{L}}_{\phi_p+\eta}$ are ample.

For every $p \in \mathbb{N}$, we have
\[
\widehat{\mathrm{vol}}(\overline{\mathcal{L}}_{\phi_p +\eta}) 
- \widehat{\mathrm{vol}}(\overline{\mathcal{L}}_{\phi_p}) 
=(n+1)!
\int_{\mathcal{X}(\mathbb{C})} 
\eta \left( \int_0^1 \mu_{\phi_p+s\eta} \, ds \right).
\]
By continuity of the arithmetic volume and 
\cite[Proposition 4.12]{BermanBoucksom} or 
\cite[Theorem 2.17]{BEGZ}, we obtain
\[
\widehat{\mathrm{vol}}(\overline{\mathcal{L}}_{\phi_\infty +\eta}) 
- \widehat{\mathrm{vol}}(\overline{\mathcal{L}}_{\phi_\infty}) 
=(n+1)!
\int_{\mathcal{X}(\mathbb{C})} 
\eta \left( \int_0^1 \mu_{\phi_\infty+s\eta} \, ds \right).
\]

\medskip

The essential minimum of $\overline{\mathcal{L}}_{\phi_\infty}$ is defined by
\[
e_1(\overline{\mathcal{L}}_{\phi_\infty}) 
:=
\sup_{\substack{U \subset \mathcal{X} \\ U \text{ open}}}
\inf_{x \in U(\overline{\mathbb{Q}})}
h_{\overline{\mathcal{L}}_{\phi_\infty}}(x).
\]
We have
\[
e_1(\overline{\mathcal{L}}_{\phi_\infty})
\geq 
\widehat{\mu}_{\max}(\overline{\mathcal{L}}_{\phi_\infty}).
\]
In the dynamical setting, one shows that 
$e_1(\overline{\mathcal{L}}_{\phi_\infty})=0$. 
Since 
$\widehat{\mu}_{\max}(\overline{\mathcal{L}}_{\phi_\infty}) \ge 0$, 
it follows that
\[
\widehat{\mu}_{\max}(\overline{\mathcal{L}}_{\phi_\infty}) = 0,
\quad
\text{hence}
\quad
\widehat{\mathrm{vol}}(\overline{\mathcal{L}}_{\phi_\infty}) = 0.
\]

\medskip

Let $(x_m)_{m \in \mathbb{N}}$ be a generic and small sequence of 
algebraic points of $\mathcal{X}$. 
Thus for every strict closed subscheme $\mathcal{Z} \subsetneq \mathcal{X}$, 
the set $\{m : x_m \in \mathcal{Z}\}$ is finite, and
\[
h_{\overline{\mathcal{L}}_{\phi_\infty}}(x_m)
\longrightarrow
 0.
\]

Let $\eta$ be a smooth nonnegative function on 
$\mathcal{X}(\mathbb{C})$. We have
\[
\liminf_{m}
h_{\overline{\mathcal{L}}_{\phi_\infty+\eta}}(x_m)
\ge 
e_1(\overline{\mathcal{L}}_{\phi_\infty+\eta}) \ge \widehat{\mu}_{\max}(\overline{\mathcal{L}}_{\phi_\infty+\eta}) 
\ge
\frac{\widehat{\mathrm{vol}}
(\overline{\mathcal{L}}_{\phi_\infty+\eta})}
{(n+1)
\,\mathrm{vol}(\mathcal{L}_{\mathbb{Q}})}.
\]

Using the definition of heights of points, we obtain
\[
\lim_{m}
h_{\overline{\mathcal{L}}_{\phi_\infty}}(x_m)
+
\liminf_{m}
\int_{\mathcal{X}(\mathbb{C})} 
\eta \, \delta_{x_m}
\ge
\frac{
\widehat{\mathrm{vol}}(\overline{\mathcal{L}}_{\phi_\infty})
+(n+1)!
\int_{\mathcal{X}(\mathbb{C})}
\eta \left( \int_0^1 
\mu_{\phi_\infty+s\eta} \, ds \right)
}
{(n+1)
\,\mathrm{vol}(\mathcal{L}_{\mathbb{Q}})},
\]
where $\delta_{x_m}$ denotes the probability measure associated to 
the Galois orbit of $x_m$.

Since 
$\widehat{\mathrm{vol}}(\overline{\mathcal{L}}_{\phi_\infty})=0$ 
and 
$h_{\overline{\mathcal{L}}_{\phi_\infty}}(x_m)\to 0$, 
we deduce
\[
\liminf_{m}
\int_{\mathcal{X}(\mathbb{C})} 
\eta \, \mu_{x_m}
\ge
\frac{n!
\int_{\mathcal{X}(\mathbb{C})}
\eta \left( \int_0^1 
\mu_{\phi_\infty+s\eta} \, ds \right)
}
{
\,\mathrm{vol}(\mathcal{L}_{\mathbb{Q}})}.
\]

Replacing $\eta$ by $\varepsilon \eta$ and letting 
$\varepsilon \to 0^+$, we obtain
\[
\liminf_{m}
\int_{\mathcal{X}(\mathbb{C})} 
\eta \, \mu_{x_m}
\ge
\frac{n!
\int_{\mathcal{X}(\mathbb{C})}
\eta \, \mu_{\phi_\infty}
}
{
\,\mathrm{vol}(\mathcal{L}_{\mathbb{Q}})}.
\]

Applying this to $M-g$, where $M=\sup g$, yields
\[
\lim_{m\to\infty}
\int_{\mathcal{X}(\mathbb{C})}
g \, \mu_{x_m}
=\frac{n!}{\mathrm{vol}(\LL_\Q)}
\int_{\mathcal{X}(\mathbb{C})}
g \, \mu_{\phi_\infty}
\qquad
\forall g \in \mathscr{C}^0(\mathcal{X}(\mathbb{C})).
\]

\bibliographystyle{plain} 

\bibliography{biblio}

\begin{thebibliography}{10}

\bibitem{Bana}
W.~Banaszczyk.
\newblock New bounds in some transference theorems in the geometry of numbers.
\newblock {\em Math. Ann.}, 296(4):625--635, 1993.

\bibitem{BermanBoucksom}
Robert Berman and S{\'e}bastien Boucksom.
\newblock Growth of balls of holomorphic sections and energy at equilibrium.
\newblock {\em Invent. Math.}, 181(2):337--394, 2010.

\bibitem{Berman2009}
Robert~J. Berman.
\newblock Bergman kernels and equilibrium measures for line bundles over
  projective manifolds.
\newblock {\em Amer. J. Math.}, 131(5):1485--1524, 2009.

\bibitem{BoGS}
J.-B. Bost, H.~Gillet, and C.~Soul{\'e}.
\newblock Heights of projective varieties and positive {G}reen forms.
\newblock {\em J. Amer. Math. Soc.}, 7(4):903--1027, 1994.

\bibitem{BostTheta}
Jean-Beno\^{\i}t Bost.
\newblock {\em Theta invariants of {E}uclidean lattices and
  infinite-dimensional {H}ermitian vector bundles over arithmetic curves},
  volume 334 of {\em Progress in Mathematics}.
\newblock Birkh\"{a}user/Springer, Cham, [2020] \copyright 2020.

\bibitem{Bouche}
Thierry Bouche.
\newblock Convergence de la m\'etrique de {F}ubini-{S}tudy d'un fibr\'e
  lin\'eaire positif.
\newblock {\em Ann. Inst. Fourier (Grenoble)}, 40(1):117--130, 1990.

\bibitem{Chen}
S{\'e}bastien Boucksom and Huayi Chen.
\newblock Okounkov bodies of filtered linear series.
\newblock {\em Compos. Math.}, 147(4):1205--1229, 2011.

\bibitem{BEGZ}
S{\'e}bastien Boucksom, Philippe Eyssidieux, Vincent Guedj, and Ahmed Zeriahi.
\newblock Monge-{A}mp\`ere equations in big cohomology classes.
\newblock {\em Acta Math.}, 205(2):199--262, 2010.

\bibitem{Burgos3}
Jos{\'e}~Ignacio Burgos~Gil, Atsushi Moriwaki, Patrice Philippon, and
  Mart{\'{\i}}n Sombra.
\newblock Arithmetic positivity on toric varieties.
\newblock {\em J. Algebraic Geom.}, 25(2):201--272, 2016.

\bibitem{Burgos_distribution}
Jos\'{e}~Ignacio Burgos~Gil, Patrice Philippon, Juan Rivera-Letelier, and
  Mart\'{\i}n Sombra.
\newblock The distribution of {G}alois orbits of points of small height in
  toric varieties.
\newblock {\em Amer. J. Math.}, 141(2):309--381, 2019.

\bibitem{Catlin}
David Catlin.
\newblock The {B}ergman kernel and a theorem of {T}ian.
\newblock In {\em Analysis and geometry in several complex variables ({K}atata,
  1997)}, Trends Math., pages 1--23. Birkh\"{a}user Boston, Boston, MA, 1999.

\bibitem{CL2006}
Antoine Chambert-Loir.
\newblock Mesures et \'equidistribution sur les espaces de {B}erkovich.
\newblock {\em J. Reine Angew. Math.}, 595:215--235, 2006.

\bibitem{Chen1}
Huayi Chen.
\newblock Arithmetic {F}ujita approximation.
\newblock {\em Ann. Sci. \'Ec. Norm. Sup\'er. (4)}, 43(4):555--578, 2010.

\bibitem{Chen-2011}
Huayi Chen.
\newblock Differentiability of the arithmetic volume function.
\newblock {\em J. Lond. Math. Soc. (2)}, 84(2):365--384, 2011.

\bibitem{DemaillyLivre}
J.P. Demailly.
\newblock {\em Complex analytic and differential geometry}.

\bibitem{finski2024wess_v2}
Siarhei Finski.
\newblock About wess--zumino--witten equation and harder--narasimhan
  potentials.
\newblock 2024.
\newblock Revised 13 Nov 2024.

\bibitem{AIT}
Henri Gillet and Christophe Soul{\'e}.
\newblock Arithmetic intersection theory.
\newblock {\em Inst. Hautes \'Etudes Sci. Publ. Math.}, 72:93--174 (1991),
  1990.

\bibitem{Character}
Henri Gillet and Christophe Soul{\'e}.
\newblock Characteristic classes for algebraic vector bundles with {H}ermitian
  metric. {I}.
\newblock {\em Ann. of Math. (2)}, 131(1):163--203, 1990.

\bibitem{Character2}
Henri Gillet and Christophe Soul{\'e}.
\newblock Characteristic classes for algebraic vector bundles with {H}ermitian
  metric. {II}.
\newblock {\em Ann. of Math. (2)}, 131(2):205--238, 1990.

\bibitem{Gubler2003}
Walter Gubler.
\newblock Local and canonical heights of subvarieties.
\newblock {\em Ann. Sc. Norm. Super. Pisa Cl. Sci. (5)}, 2(4):711--760, 2003.

\bibitem{HajliTAMS}
Mounir Hajli.
\newblock The theta invariants and the volume function on arithmetic varieties.
\newblock {\em Trans. Amer. Math. Soc.}, 376(3):2237--2256, 2023.

\bibitem{Hisamoto1}
Tomoyuki Hisamoto.
\newblock Restricted {B}ergman kernel asymptotics.
\newblock {\em Trans. Amer. Math. Soc.}, 364(7):3585--3607, 2012.

\bibitem{Hisamoto2}
Tomoyuki Hisamoto.
\newblock On the limit of spectral measures associated to a test configuration
  of a polarized {K}\"ahler manifold.
\newblock {\em J. Reine Angew. Math.}, 713:129--148, 2016.

\bibitem{Moriwaki1}
Atsushi Moriwaki.
\newblock Arithmetic height functions over finitely generated fields.
\newblock {\em Invent. Math.}, 140(1):101--142, 2000.

\bibitem{Moriwaki2}
Atsushi Moriwaki.
\newblock Continuity of volumes on arithmetic varieties.
\newblock {\em J. Algebraic Geom.}, 18(3):407--457, 2009.

\bibitem{Moriwaki}
Atsushi Moriwaki.
\newblock Big arithmetic divisors on the projective spaces over {$\Bbb Z$}.
\newblock {\em Kyoto J. Math.}, 51(3):503--534, 2011.

\bibitem{Oda}
Tadao Oda.
\newblock Convex bodies and algebraic geometry---toric varieties and
  applications. {I}.
\newblock In {\em Algebraic {G}eometry {S}eminar ({S}ingapore, 1987)}, pages
  89--94. World Sci. Publishing, Singapore, 1988.

\bibitem{convex}
R.~Tyrrell Rockafellar.
\newblock {\em Convex analysis}.
\newblock Princeton Landmarks in Mathematics. Princeton University Press,
  Princeton, NJ, 1997.
\newblock Reprint of the 1970 original, Princeton Paperbacks.

\bibitem{SUZ}
L.~Szpiro, E.~Ullmo, and S.~Zhang.
\newblock \'equir\'epartition des petits points.
\newblock {\em Invent. Math.}, 127(2):337--347, 1997.

\bibitem{Tian}
Gang Tian.
\newblock On a set of polarized {K}\"ahler metrics on algebraic manifolds.
\newblock {\em J. Differential Geom.}, 32(1):99--130, 1990.

\bibitem{Yuan}
Yuan Xinyi.
\newblock On volumes of arithmetic line bundles ii.
\newblock {\em arXiv:0909.3680v1}.

\bibitem{YuanInventiones}
Xinyi Yuan.
\newblock Big line bundles over arithmetic varieties.
\newblock {\em Invent. Math.}, 173(3):603--649, 2008.

\bibitem{Yuan_2009}
Xinyi Yuan.
\newblock On volumes of arithmetic line bundles.
\newblock {\em Compos. Math.}, 145(6):1447--1464, 2009.

\bibitem{YuanZhang1}
Xinyi Yuan and Tong Zhang.
\newblock Effective bound of linear series on arithmetic surfaces.
\newblock {\em Duke Math. J.}, 162(10):1723--1770, 2013.

\bibitem{YuanZhang2}
Xinyi Yuan and Tong Zhang.
\newblock Effective bounds of linear series on algebraic varieties and
  arithmetic varieties.
\newblock {\em J. Reine Angew. Math.}, 736:255--284, 2018.

\bibitem{Zelditch}
Steve Zelditch.
\newblock Szeg{o} kernels and a theorem of {T}ian.
\newblock {\em Internat. Math. Res. Notices}, (6):317--331, 1998.

\bibitem{Zha95}
Shouwu Zhang.
\newblock Positive line bundles on arithmetic varieties.
\newblock {\em J. Amer. Math. Soc.}, 8(1):187--221, 1995.

\bibitem{Zhang}
Shouwu Zhang.
\newblock Small points and adelic metrics.
\newblock {\em J. Algebraic Geom.}, 4(2):281--300, 1995.

\end{thebibliography}

\end{document}